\documentclass[12pt,francais]{article}
\usepackage{amsmath}\usepackage{epsf,amsfonts,amsthm}\usepackage{amscd,amssymb}
\usepackage{xcolor,epic,eepic}\usepackage{epsfig}
\usepackage{indentfirst, delarray}
\usepackage{rotating}
\usepackage{mathdots}
\usepackage[matrix,arrow,curve]{xy}
\usepackage{graphicx}
\usepackage{fancyhdr}
\usepackage{dsfont,texdraw}
\usepackage{mathrsfs}
\usepackage{latexsym}
\usepackage{hyperref}
\usepackage{tikz-cd}
\theoremstyle{plain}

\usepackage{tikz}
\usetikzlibrary{arrows,chains,matrix,positioning,scopes,snakes}
\makeatletter
\tikzset{join/.code=\tikzset{after node path={%
\ifx\tikzchainprevious\pgfutil@empty\else(\tikzchainprevious)%
edge[every join]#1(\tikzchaincurrent)\fi}}}
\makeatother

\tikzset{>=stealth',every on chain/.append style={join},
         every join/.style={->}}
\tikzset{
    >=stealth',
    punkt/.style={
           rectangle,
           rounded corners,
           draw=black, very thick,
           text width=6.5em,
           minimum height=2em,
           text centered},
    pil/.style={
           ->,
           thick,
           shorten <=2pt,
           shorten >=2pt,}
}

\newcommand{\bee}{\begin{enumerate}}
\newcommand{\eee}{\end{enumerate}}
\newcommand{\benn}{\begin{equation*}}
\newcommand{\eenn}{\end{equation*}}
\newcommand{\be}{\begin{equation}}
\newcommand{\ee}{\end{equation}}
\newcommand{\bean}{\begin{eqnarray}}
\newcommand{\eean}{\end{eqnarray}}
\newcommand{\bea}{\begin{eqnarray*}}
\newcommand{\eea}{\end{eqnarray*}}

\newcommand{\Z}{\mathbb{Z}}

\newcommand{\cF}{{\cal F}}

\newcommand{\op}[1]{\!\!\mathop{\rm ~#1}\nolimits}
\newcommand{\id}{\op{id}}

\newcommand{\cL}{{\cal L}}

\mathchardef\za="710B  
\mathchardef\zb="710C  
\mathchardef\zg="710D  
\mathchardef\zd="710E  
\mathchardef\zve="710F 
\mathchardef\zz="7110  
\mathchardef\zh="7111  
\mathchardef\zy="7112 

\mathchardef\zi="7113  
\mathchardef\zk="7114  
\mathchardef\zl="7115  
\mathchardef\zm="7116  
\mathchardef\zn="7117  
\mathchardef\zx="7118  
\mathchardef\zp="7119  
\mathchardef\zr="711A  
\mathchardef\zs="711B  
\mathchardef\zt="711C  
\mathchardef\zu="711D  
\mathchardef\zf="711E 
\mathchardef\zq="711F  
\mathchardef\zc="7120  
\mathchardef\zw="7121  
\mathchardef\ze="7122  
\mathchardef\zvy="7123  
\mathchardef\zvw="7124  
\mathchardef\zvr="7125 
\mathchardef\zvs="7126 
\mathchardef\zvf="7127  
\mathchardef\zG="7000  
\mathchardef\zD="7001  
\mathchardef\zY="7002  
\mathchardef\zL="7003  
\mathchardef\zX="7004  
\mathchardef\zP="7005  
\mathchardef\zS="7006  
\mathchardef\zU="7007  
\mathchardef\zF="7008  
\mathchardef\zW="700A  

 \newcommand{\cH}{{\cal H}}
 
 \newcommand{\cC}{{\cal C}}

 \newcommand{\cD}{{\cal D}}

 \newcommand{\cQ}{{\cal Q}}
 \newcommand{\cI}{{\cal I}}

 \newcommand{\cG}{{\cal G}}
 \newcommand{\cJ}{{\cal J}}

\newtheorem{rem}{Remark}
\newtheorem{theo}{Theorem}
\newtheorem{prop}{Proposition}
\newtheorem{lem}{Lemma}
\newtheorem{cor}{Corollary}

\newtheorem{defi}{Definition}

\newcommand{\Hom}{\mathrm{Hom}}

\begin{document}

\title{Comparison theorems for Kan, faintly universal and strongly universal derived functors}
\author{Alisa Govzmann, Damjan Pi\v{s}talo, and Norbert Poncin}
\maketitle

\begin{abstract} We distinguish between faint, weak, strong and strict localizations of categories at morphism families and show that this framework captures the different types of derived functors that are considered in the literature. More precisely, we show that Kan and faint derived functors coincide when we use the classical Kan homotopy category, and when we use the Quillen homotopy category, Kan and strong derived functors coincide. Our comparison results are based on the fact that the Kan homotopy category is a weak localization and that the Quillen homotopy category is a strict localization. \end{abstract}

\small{\vspace{2mm} \noindent {\bf MSC 2020}: 18E35, 18N40, 14A30 \medskip

}

\tableofcontents

\thispagestyle{empty}

\section{Introduction}

There are a number of definitions of a localization of a category at a family of morphisms, of definitions of a model category, the homotopy category and the derived functor in the literature, and it is not easy to navigate this jungle of different concepts. In the present text we unravel this tangle.\medskip

In Section \ref{CatWeq} we define and liken four types of more or less tight localization of a category, which turn out to provide the proper framework for the study and comparison of various derived functors in model categories. We refer to these localizations at a distinguished class of morphisms as the faint, weak, strong and strict localizations.\medskip

To define derived functors on homotopy categories most sources use the strong localization of the model category at its weak equivalences given by the Quillen homotopy category (the homotopy category for short) \cite{DS}, \cite{Hir}, \cite{Ho99}; others work with the faint localization given by the Kan homotopy category (the classical homotopy category to avoid ambiguity) \cite{JL}, nLab. Actually the Kan and Quillen homotopy categories are an example of a weak (hence a faint) and a strict (hence a strong) localization, respectively. Although the strictness of the Quillen localization implies immediately that it is also a weak localization and therefore necessarily equivalent to the Kan localization, this property of the homotopy category is to the best of our knowledge not mentioned in the literature. We deal with the preceding aspects in Section \ref{HoCats}.\medskip

In Section \ref{DerFuns} we define derived functors as Kan extensions along localization functors (K derived functors) and as factorizations through faint and strong localizations of the source category at its weak equivalences (F and S derived functors). K derived functors can be defined along every functor, in particular along the localization functor of the Quillen homotopy category \cite{DS}, \cite{Hir}, \cite{Quill} and the localization functor of the Kan homotopy category, nLab. F derived functors will be defined for the faint localization functor of the latter \cite{JL}, nLab, and S derived functors for the strong localization functor of the former \cite{Ho99}. The reader will observe that the fact that these localization functors are actually weak and strict, respectively, does not play any role in the derived functors' existence and uniqueness results and he will notice that the definitions of K, F and S derived functors differ in particular by the strength of their `commutation' with the localization functor considered. We close Section \ref{DerFuns} by proving that in the case of the Kan homotopy category K and F derived functors coincide and that in the case of the Quillen homotopy category K and S derived functors are the same. While the weakness of the Kan localization has not been exploited so far, it is important for proving the previous comparison theorem. Although nothing is really new, neither of the comparison theorems seems to exist in the literature.\medskip

In the final section \ref{FutDir}, we briefly describe the context that led to the need to compare K, F, and S derived functors.\medskip

The proven comparison theorems should be of interest to any research problem that requires results from sources using different definitions of derived functors. Applications can be expected for instance in higher geometry and physics. Indeed, functors and their derivatives are of importance in environments where there is no good notion of point, e.g., in supergeometry and in algebraic geometry: higher supergeometry \cite{ Berezinian, Z2nManifolds, LocalForms,Integration}, homotopical algebraic geometry \cite{TV05, TV08} and its generalisation that goes under the name of homotopical algebraic geometry over differential operators, are completely based on the functor of points approach \cite{Schwarz, Linear, HAC}.\medskip

{\bf Warning}. In this text we identify objects that are connected by a unique or canonical isomorphism.

\section{Strict, strong, weak and faint localizations}\label{CatWeq}

If one looks at the details of the various definitions of derived functors of Quillen functors, one sees that the authors essentially use two different definitions of localization of a category at a family of morphisms. We will refer to the notion of localization that is used in \cite{DS}, \cite{Hir} and \cite{Ho99} as the {\it strong localization} and to the notion of localization that is used in \cite{JL} as a {\it faint localization}. However:

\begin{rem} There are two additional types of localization, a specific strong localization (called strict localization) and a specific faint localization (called weak localization), which may represent better localization concepts, although their distinctive specificity is not needed in the definitions of derived functors -- on the other hand it is of fundamental importance in the proofs of the comparison theorems of these definitions.
\end{rem}

\begin{defi}\label{Faint}
A {\bf faint localization} of a category $\tt C$ at a family $W$ of morphisms (called weak equivalences) is a category ${\tt C}[W^{-1}]$ together with a functor $L:{\tt C}\to {\tt C}[W^{-1}]$ that sends weak equivalences to isomorphisms, such that the pair $({\tt C}[W^{-1}]),L)$ is {\it faintly universal} in the sense that:

\begin{enumerate}
  \item[\emph{(L1)}] If $({\tt D},F)$ is another such pair, there exists a functor $\tilde{F}:{\tt C}[W^{-1}]\to{\tt D}$ such that the resulting triangle commutes (not on the nose but) up to natural isomorphism $\zh:F\stackrel{\cong}{\Rightarrow}\tilde{F}\circ L\,$.
      \begin{equation} \begin{tikzpicture}
 \matrix (m) [matrix of math nodes, row sep=3em, column sep=3em]
   {  {\tt C}  & {\tt C}[W^{-1}]  \\
       & {\tt D}  \\ };
 \path[->]
 (m-1-1) edge node[left] {\small{$F$}} (m-2-2)
 (m-1-1) edge node[above] {\small{$L$}} (m-1-2)
 (m-1-2) edge [->, dashed] node[auto] {\small{$\tilde{F}$}} (m-2-2);
 \draw[-latex] node[auto]{$\;\;\zh$} node[above,rotate=45]{$\;\;\;\;\stackrel{\cong}{\Rightarrow}$};
\end{tikzpicture}
\end{equation}
  \item [\emph{(L2)}] The pair $(\tilde{F},\zh)$ is unique up to (!) unique natural isomorphism, i.e., if $(\tilde{F}',\zh')$ is another such pair, there is a unique natural isomorphism $\ze:\tilde{F}\stackrel{\cong}{\Rightarrow} \tilde{F}'$ such that the obvious diagram commutes on the nose: $(\ze\star L)\circ\zh=\zh'\,,$ where $\star$ denotes whiskering.
\end{enumerate}
\end{defi}

\begin{defi}\label{Weak}
A {\bf weak localization} of a category $\tt C$ at a family $W$ of morphisms is a category ${\tt C}[W^{-1}]$ together with a functor $L:{\tt C}\to {\tt C}[W^{-1}]$ that sends weak equivalences to isomorphisms, such that the pair $({\tt C}[W^{-1}]),L)$ is {\it weakly universal} in the sense that:

\begin{enumerate}
  \item[\emph{(L1)}] If $({\tt D},F)$ is another such pair, there exists a functor $\tilde{F}:{\tt C}[W^{-1}]\to{\tt D}$ such that the resulting triangle commutes up to natural isomorphism $\zh:F\stackrel{\cong}{\Rightarrow}\tilde{F}\circ L\,$.
  \item[\emph{(L2')}] The functor $$-\circ L:{\tt Fun(C}[W^{-1}],{\tt D})\to {\tt Fun}{\tt (C,D)}$$ from functors out of ${\tt C}[W^{-1}]$ to functors out of $\tt C$ is fully faithful for every category $\tt D\,$.
\end{enumerate}
\end{defi}

\begin{defi}\label{Strong}
A {\bf strong localization} of a category $\tt C$ at a family $W$ of morphisms is a category ${\tt C}[[W^{-1}]]$ together with a functor $L:{\tt C}\to {\tt C}[[W^{-1}]]$ that sends weak equivalences to isomorphisms, such that the pair $({\tt C}[[W^{-1}]]),L)$ is {\it strongly universal} in the sense that:\medskip

\begin{enumerate}
\item[\emph{(L1')}]If $({\tt D},F)$ is another such pair, there exists a unique functor $\tilde{F}:{\tt C}[[W^{-1}]]\to{\tt D}$ such that the resulting triangle commutes on the nose: $F=\tilde{F}\circ L\,$.
\end{enumerate}
\end{defi}

\begin{defi}\label{Strict}
A {\bf strict localization} of a category $\tt C$ at a family $W$ of morphisms is a category ${\tt C}[[W^{-1}]]$ together with a functor $L:{\tt C}\to {\tt C}[[W^{-1}]]$ that sends weak equivalences to isomorphisms, such that the pair $({\tt C}[[W^{-1}]]),L)$ is {\it strictly universal} in the sense that:

\begin{enumerate}
  \item[\emph{(L1')}] If $({\tt D},F)$ is another such pair, there exists a unique functor $\tilde{F}:{\tt C}[[W^{-1}]]\to{\tt D}$ such that the resulting triangle commutes on the nose: $F=\tilde{F}\circ L\,$.
  \item[\emph{(L2')}] The functor $$-\circ L:{\tt Fun(C}[[W^{-1}]],{\tt D})\to {\tt Fun}{\tt (C,D)}$$ from functors out of ${\tt C}[[W^{-1}]]$ to functors out of $\tt C$ is fully faithful for every category $\tt D\,$.
\end{enumerate}
\end{defi}

As mentioned before, the concept of faint (resp., strong) localization is used in \cite{JL} (resp., \cite{DS}, \cite{Hir}, \cite{Ho99}). The concept of weak localization can be found for instance in \cite{KS}.\medskip

If $\tt CW$ denotes either ${\tt C}[W^{-1}]$ or ${\tt C}[[W^{-1}]]$ the fully faithfulness condition in Definitions \ref{Weak} and \ref{Strict} means that for every functors $F,G\in\tt Fun(CW,D)$ the map \be\label{FulFai}(-\circ L)_{F,G}:\op{Hom}_{\tt Fun(CW,D)}(F,G)\ni \zz\mapsto \zz\star L\in\op{Hom}_{\tt Fun(C,D)}(F\circ L,G\circ L)\ee is bijective.

\begin{prop}\label{SW}
Every strict localization is a weak localization.
\end{prop}

\begin{proof}
Obvious.
\end{proof}

\begin{prop}\label{WFSS}
Every weak (resp., strict) localization is a faint (resp., strong) localization.
\end{prop}

\begin{proof}
The second statement is obvious. Let now $({\tt C}[W^{-1}],L)$ be a weak localization and let $({\tt D},F)\,,$ $(\tilde{F},\zh)$ and $(\tilde{F}',\zh')$ be as in Definition \ref{Faint}. We must show that there is a unique natural isomorphism $\ze:\tilde{F}\stackrel{\cong}{\Rightarrow}\tilde{F}'$ such that $\ze\star L=\zh'\circ\zh^{-1}\,.$ Since $\zh'\circ\zh^{-1}:\tilde{F}\circ L\stackrel{\cong}{\Rightarrow}\tilde{F}'\circ L$ and similarly for $\zh\circ\zh'^{-1}\,,$ it follows from the full faithfulness \eqref{FulFai} of $-\circ L$ that there exists a unique natural transformation $\zy:\tilde{F}\Rightarrow\tilde{F}'$ such that $\zy\star L=\zh'\circ\zh^{-1}$ and a unique natural transformation $\zvy:\tilde{F}'\Rightarrow\tilde{F}$ such that $\zvy\star L=\zh\circ\zh'^{-1}\,.$ Hence the bijection $-\circ L$ maps $\zvy\circ\zy:\tilde{F}\Rightarrow\tilde{F}$ to $$(\zvy\circ\zy)\star L=(\zvy\star L)\circ(\zy\star L)=\id_{\tilde{F}\circ L}=\id_{\tilde F}\star L\;,$$ so that $\zvy\circ\zy=\id_{\tilde{F}}\,.$ Similarly one gets that $\zy\circ\zvy=\id_{\tilde{F}'}\,,$ which completes the proof.
\end{proof}

Let $[{\tt C,D}]_W$ be the full subcategory of those functors of $\tt Fun(C,D)$ that send the morphisms in $W$ to isomorphisms in $\tt D\,.$ Observe that Condition (L1') means that any functor of $$[{\tt C,D}]_W$$ factors uniquely through $L\,,$ and that in view of \eqref{FulFai} Condition (L2') means that any natural transformation $$\zx:F\circ L\Rightarrow G\circ L$$ in $$\op{Hom}_{[{\tt C,D}]_W}(F\circ L, G\circ L)$$ factors uniquely through $L\,.$ Hence a strong localization needs not be strict. Noticing that Condition (L2) means that any natural isomorphism of the type $$\zh'\circ\zh^{-1}:\tilde{F}\circ L\stackrel{\cong}{\Rightarrow}\tilde{F}'\circ L$$ factors uniquely through $L$ as isomorphism, we see that a faint localization needs not be weak.

\begin{prop}\label{WFSSNSC} Let $\tt C$ be a category and $W\!$ a family of morphisms of $\,\tt C\,.$

\begin{enumerate}
\item A faint localization $({\tt C}[W^{-1}],L)$ is a weak localization if and only if for every $\tt D$ the functor \be\label{ECF}-\circ L:{\tt Fun(C}[W^{-1}],{\tt D})\to [{\tt C,D}]_W\ee is an equivalence of categories.
\item A strong localization $({\tt C}[[W^{-1}]],L)$ is a strict localization if and only if for every $\tt D$ the functor \be\label{ICF}-\circ L:{\tt Fun(C}[[W^{-1}]],{\tt D})\to [{\tt C,D}]_W\ee is an isomorphism of categories.
\end{enumerate}
\end{prop}

\begin{proof}
Like every functor, the functor \eqref{ECF} yields an equivalence of categories if and only if it is fully faithful and essentially surjective. As we are dealing with a faint localization every object $F\in[{\tt C,D}]_W$ is naturally isomorphic to some $\tilde{F}\circ L\,,$ so that $-\circ L$ is essentially surjective. Hence it is an equivalence of categories if and only if it is fully faithful, i.e., if and only if it is fully faithful as $\tt Fun(C,D)$-valued functor, i.e., if and only if the localization considered is a weak localization.\medskip

The functor \eqref{ICF} yields an isomorphism of categories if and only if it is bijective on objects and on morphisms. Since the localization considered is strong, the functor $-\circ L$ is bijective on objects. Hence it is an isomorphism of categories if and only if the localization is strict.
\end{proof}

\begin{prop}\label{UEUI}
If a faint (resp., strong) localization exists it is unique up to equivalence of categories (resp., up to unique isomorphism of categories). Because of Proposition \ref{WFSS}, this is especially true for a weak (resp., strict) localization.
\end{prop}

\begin{proof}
Let $({\tt C}[W^{-1}],L)$ and $({\tt C}[W^{-1}]',L')$ be two faint localizations of $\tt C$ at $W\,.$ Using the faint universality of the first pair with respect to the second, we get a functor $$\tilde{L}':{\tt C}[W^{-1}]\to{\tt C}[W^{-1}]'$$ and a natural isomorphism $\zh':L'\stackrel{\cong}{\Rightarrow}\tilde{L}'\circ L\;.$ Dually we obtain a functor $$\tilde{L}:{\tt C}[W^{-1}]'\to{\tt C}[W^{-1}]$$ and a natural isomorphism \be\label{LInd1}\zh:L\stackrel{\cong}{\Rightarrow}\tilde{L}\circ L'\;.\ee Both pairs $(\tilde{L}',\zh')$ and $(\tilde{L},\zh)$ are unique up to unique isomorphism. As the resulting functor $$\tilde{L}\circ\tilde{L}':{\tt C}[W^{-1}]\to {\tt C}[W^{-1}]$$ and natural isomorphism $$( \tilde{L}\star\zh')\circ\zh:L\stackrel{\cong}{\Rightarrow}\tilde{L}\circ\tilde{L}'\circ L$$ are unique up to unique natural isomorphism due to the faint universality of the first pair with respect to itself, and as the identity functor of ${\tt C}[W^{-1}]$ and the identity natural isomorphism of $L$ are a second pair of this type, there exists a unique natural isomorphism \be\label{LInd2}\tilde{L}\circ\tilde{L}'\stackrel{\cong}{\Rightarrow}\id_{{\tt C}[W^{-1}]}\ee such that the obvious diagram commutes. Dually we get a unique natural isomorphism \be\label{LInd3}\tilde{L}'\circ\tilde{L}\stackrel{\cong}{\Rightarrow}\id_{{\tt C}[W^{-1}]'}\;.\ee

In the case of the strong localization the proof is the same, but all natural isomorphisms are equalities and the inverse categorical isomorphisms $\tilde{L}'$ and $\tilde{L}$ are unique.
\end{proof}

\section{Weak and strict homotopy categories}\label{HoCats}

\subsection{Various replacements in a model category}\label{Rep}

We assume that the reader is familiar with {\it model categories}. Essentially a model category is a category which comes equipped with three classes of morphisms: weak equivalences, fibrations and cofibrations. A fundamental concept is homotopies between two maps between the same objects. Whitehead's theorem states that a weak equivalence between fibrant-cofibrant objects can be inverted up to homotopy and more precisely that a map between fibrant-cofibrant objects is a weak equivalence if and only if it is a homotopy equivalence.\medskip

Recall that a model category admits a cofibration - trivial fibration factorization $(a,b)$ ({\small Cof - TrivFib} factorization) and a trivial cofibration - fibration factorization $(a',b')$ ({\small TrivCof - Fib} factorization) \cite{JL}. Often this system of factorizations is required to be a functorial factorization system $(\za,\zb), (\za',\zb')$ \cite{Hir}.\medskip

{\bf Warning}. \emph{In this text we use exclusively model categories that admit a functorial factorization system}.\medskip

Another variant of the definition of a model category not only requires the existence of a functorial system, but fixes such a system and views it as part of the model structure; changing the system leaving everything else unchanged, leads to an isomorphic model category in all reasonable senses of `isomorphism of model categories' \cite{Ho99}.\medskip

Let $\tt M$ be a model category. We denote its initial and terminal objects by $0$ and $\ast$ respectively.\medskip

Let now $(\za,\zb), (\za',\zb')$ be any functorial factorization system. For every object $X\in\tt M\,,$ the first factorization factors the map $i_X:0\to X$ into a cofibration $\za(i_X)$ followed by a trivial fibration $q_X:=\zb(i_X)\,$: $$i_X:0\rightarrowtail QX\stackrel{\sim}{\twoheadrightarrow}X\;.$$ Regardless of the factorization $$i_X:0\rightarrowtail CX\stackrel{\sim}{\to}X$$ of $i_X:0\to X$ into a cofibration followed by a weak equivalence $c_X$ considered, we refer to $CX$ as {\it a} {\it cofibrant replacement} of $X\,.$ The object $QX$ we call a {\it cofibrant F-replacement} of $X$ (or just a cofibrant replacement if we do not want to stress that $q_X$ is a fibration). From the fact that the factorization $(\za,\zb)$ is functorial it follows that $Q$ is an endofunctor of $\tt M\,$. Moreover $q_X:QX\to X$ is functorial in $X\,$: $q$ is a natural transformation $q:Q\Rightarrow\op{id}_{\tt M}$ from the {\it cofibrant replacement functor} $Q$ to the identity functor $\op{id}_{\tt M}$ \cite{Ho99}. Instead of the cofibrant F-replacement functor $Q$ that is globally defined by the functorial factorization $(\za,\zb)\,,$ we will also use local / object-wise cofibrant replacements $CX$ or {\it local cofibrant F-replacements} $\tilde{C}X$ such that the map $c_X$ in the factorization $$i_X:0\rightarrowtail \tilde{C}X\stackrel{\sim}{\twoheadrightarrow}X$$ is $\op{id}_X$ if $X$ is already cofibrant \cite{Quill}. If for every $X$ we choose such a local cofibrant F-replacement and if $f:X\to Y\,,$ there is a lifting $\tilde{C}f:\tilde{C}X\to \tilde{C}Y,$ which will play an important role.

\begin{equation}\label{CTilde} \begin{tikzpicture}
 \matrix (m) [matrix of math nodes, row sep=3em, column sep=3em]
   {  \stackrel{}{0}  & & \tilde{C}Y  \\
      \tilde{C}X & \stackrel{}{X} & \stackrel{}{Y}  \\ };
 \path[->]
 (m-1-1) edge [>->] (m-1-3)
 (m-1-1) edge [>->] (m-2-1)
 (m-2-1) edge [->>] node[below] {\small{$\;\;{}_{\widetilde{}}\,\;c_X$}} (m-2-2)
 (m-1-3) edge [->>] node[auto] {\small{$\;{}_{\widetilde{}}\;\;c_Y$}}(m-2-3)
 (m-2-2) edge node[below] {\small{$f$}} (m-2-3)
 (m-2-1) edge [->, dashed] node[auto] {\small{$\tilde{C}f$}} (m-1-3);
\end{tikzpicture}
\end{equation}

The dual concepts of {\it fibrant replacement} $FX\,,$ of {\it fibrant $C$-replacement} $RX\,,$ {\it fibrant replacement functor} $R$ with natural transformation $r:\id_{\tt M}\Rightarrow R\,$, and of {\it local fibrant $C$-replacements} $f_X:X\stackrel{\sim}{\rightarrowtail}\tilde{F}X$ such that $f_X$ is identity if $X$ is a fibrant object are defined similarly using the functorial factorization $(\za',\zb')$ and the map $t_X:X\to \ast\,.$

\begin{equation}\label{FTilde} \begin{tikzpicture}
 \matrix (m) [matrix of math nodes, row sep=3em, column sep=3em]
   {  \stackrel{}{X}  & \stackrel{}{Y} & \tilde{F}Y  \\
      \tilde{F}X & & \stackrel{}{\ast}  \\ };
 \path[->]
 (m-1-1) edge node[auto] {\small{$f$}} (m-1-2)
 (m-1-2) edge [>->] node[auto] {\small{$\;\;{}_{\widetilde{}}\,\;f_Y$}}(m-1-3)
 (m-2-1) edge [->>] (m-2-3)
 (m-1-1) edge [>->] node[left] {\small{$\;\;{}_{\widetilde{}}\,\;f_X$}} (m-2-1)
 (m-1-3) edge [->>] (m-2-3)
 (m-2-1) edge [->, dashed] node[below] {\small{$\tilde{F}f$}} (m-1-3);
\end{tikzpicture}
\end{equation}

\subsection{Homotopy categories of a model category}

Let $\tt M$ be a model category. We denote its class of weak equivalences $f:X\stackrel{\sim}{\to} Y$ by $W$. Further, we write $f: X\stackrel{\cong}{\to} Y$ if $f$ is an isomorphism and we write $f\simeq^\ell g$ (resp., $f\simeq^r g$, $f\simeq g$) if $f,g:X\to Y$ are left homotopic (resp., right homotopic, homotopic). The homotopy class of a morphism $f:X\to Y$ will be denoted by $[f]\,.$ Recall also that the full subcategory ${\tt M}_{\op{fc}}\subset{\tt M}$ of fibrant-cofibrant objects of $\tt M\,,$ the full subcategory ${\tt M}_{\op{f}}\subset {\tt M}$ of fibrant objects and the full subcategory ${\tt M}_{\op{c}}\subset {\tt M}$ of cofibrant objects are all three categories with weak equivalences, whose weak equivalences $W$ are the weak equivalences $W$ of $\tt M$ between their objects. Further $\tt M_{\op{f}}$ (resp., $\tt M_{\op{c}}$) inherits fibrations (resp., cofibrations) and is the prototypical example of a fibration (resp., cofibration) category: fibration (resp., cofibration) categories have a notion weak equivalences and of fibrations (resp., cofibrations), but none of cofibrations (resp., fibrations). Hence ${\tt M}_{\op{fc}},$ ${\tt M}_{\op{f}}$ and ${\tt M}_{\op{c}}$ are {\it not} full model categories. Finally, we will in the following repeatedly use the canonical inclusion functor of ${\tt M}_{\op{c}}\subset{\tt M}$ and will denote it $i:{\tt M}_{\op{c}}\hookrightarrow{\tt M}\,$.

\subsubsection{Weak homotopy category}

\begin{defi}
The {\bf Kan homotopy category} ${\tt Ho}_{\op{K}}({\tt M})$ of a model category $\tt M$ is the category whose objects are the objects of ${\tt M}_{\op{fc}}$ and whose morphisms are the homotopy classes of morphisms of ${\tt M}_{\op{fc}}\,$\emph{:} $$\Hom_{{\tt Ho}_{\op{K}}({\tt M})}(X,Y):=\Hom_{\tt M}(X,Y)/\simeq\quad (X,Y\in{\tt M}_{\op{fc}})\;.$$
\end{defi}

\begin{theo}\label{KanWeak} The Kan homotopy category ${\tt Ho}_{\op{K}}({\tt M})$ is a weak localization ${\tt M}[W^{-1}]$ of $\,\tt M$ at its class $W$ of weak equivalences. The localization functor $\cL:{\tt M}\to {\tt Ho}_{\op{K}}({\tt M})$ is defined by $\cL X:=\tilde{F}\tilde{C}X$ on objects $X$ and on morphisms $f:X\to Y$ by $\cL f:=[\tilde{F}\tilde{C}f]\,.$ Here $\tilde{F}$ refers to a local fibrant C-replacement and $\tilde{C}$ to a local cofibrant F-replacement in the sense of Subsection \ref{Rep}.
\end{theo}

The proof is based on homotopy lemmas. We recall that two $\tt M$-morphisms $f,g:X\to Y$ between the same $\tt M$-objects are left (resp., right) homotopic and we write $f\simeq^{\ell}g$ (resp., $f\simeq^r g$) if $f\amalg g:X\amalg X\to Y$ (resp., $(f,g):X\to Y\times Y$) factors trough a cylinder object of $X$ (resp., a path object of $Y$). A cylinder object $\op{Cyl}(X)$ of $X$ is a factorization $$X\amalg X\rightarrowtail\op{Cyl}(X)\stackrel{\sim}{\to}X$$ of the fold map $\id_X\amalg\id_X:X\amalg X\to X$ into a cofibration $i$ followed by a weak equivalence $w$. The left homotopy factorization of $f\amalg g$ now means that there is a morphism $H:\op{Cyl}(X)\to Y$ such that $$f\amalg g=H\circ i\;.$$ If we denote the morphisms that come with the coproduct $X\amalg X$ by $\zf_1,\zf_2:X\rightrightarrows X\amalg X\,,$ the left homotopy factorization reads $$H\circ i_1:=H\circ i\circ\zf_1=f\quad\text{and}\quad H\circ i_2:=H\circ i\circ\zf_2=g\;.$$ Remember as well that left (resp., right) composition preserves left (resp., right) homotopies and that left (resp., right) homotopic maps are also right (resp., left) homotopic if the source (resp., target) object is cofibrant (resp., fibrant). Recall finally that if $X$ is cofibrant and $\zg:Y\stackrel{\sim}{\twoheadrightarrow}Z$ is a trivial fibration, then left composition by $\zg$ induces a 1:1 correspondence between left homotopy classes of morphisms: \be\label{1:1HC}\zg\circ -:\op{Hom}_{\tt M}(X,Y)/\simeq^{\ell}\;\to\; \op{Hom}_{\tt M}(X,Z)/\simeq^{\ell}\;.\ee The map $\zg\circ -$ is indeed well defined given the homotopy conservation property of composition mentioned above. Further, it is obviously surjective due to the lifting axiom of model categories. As for injectivity, let $f,g:X\to Y$ and assume that $\zg\circ f,\zg\circ g:X\to Z$ are left homotopic, i.e., that there is a morphism $H:\op{Cyl}(X)\to Z$ such that $\zg\circ(f\amalg g)=H\circ i\,.$ Since $\zg$ is a trivial fibration and $i$ a cofibration, the lifting axiom gives a morphism $\cH:\op{Cyl}(X)\to Y$ such that $f\amalg g=\cH\circ i\,.$ The dual result of \eqref{1:1HC} holds likewise.
\be
\begin{tikzcd}
X\amalg X\ar[r,"f\amalg\, g"]\ar[d,tail,swap,"i"]&Y\ar[d,two heads,"\zg"]\ar[d,swap,"\sim"]\\
\op{Cyl}(X)\ar[r,swap,"H"]\ar[ru,dashed,"\cH"]&Z
\end{tikzcd}
\ee

\begin{proof}
Since $0\rightarrowtail \tilde{C}X\stackrel{\sim}{\rightarrowtail}\tilde{F}\tilde{C}X\twoheadrightarrow\ast\,,$ the value $\cL X$ is fibrant and cofibrant and is therefore an object of ${\tt Ho}_{\op{K}}({\tt M})\,.$ As for $\cL f\,,$ notice that in view of \eqref{CTilde} and \eqref{FTilde} the liftings $\tilde{C}f$ and $\tilde{F}\tilde{C}f$ are any $\tt M$-morphisms that render the following squares commutative:

\begin{equation}\label{CFTilde} \begin{tikzpicture}
 \matrix (m) [matrix of math nodes, row sep=3em, column sep=3em]
   {  \stackrel{}{X}  & \tilde{C}X & \tilde{F}\tilde{C}X  \\
      \stackrel{}{Y} & \tilde{C}Y & \tilde{F}\tilde{C}Y  \\ };
 \path[->]
 (m-1-2) edge [->>] node[above] {\small{$\;\;{}_{\widetilde{}}\,\;c_X$}} (m-1-1)
 (m-1-2) edge [>->] node[auto] {\small{$\;\;{}_{\widetilde{}}\,\;f_{\tilde{C}X}$}}(m-1-3)
 (m-2-2) edge [->>] node[auto] {\small{$\;\;{}_{\widetilde{}}\,\;c_Y$}} (m-2-1)
 (m-2-2) edge [>->] node[below] {\small{$\;\;{}_{\widetilde{}}\,\;f_{\tilde{C}Y}$}}(m-2-3)
 (m-1-1) edge [->] node[auto] {$f$} (m-2-1)
 (m-1-2) edge [->] node[auto] {$\tilde{C}f$} (m-2-2)
 (m-1-3) edge [->] node[auto] {$\tilde{F}\tilde{C}f$} (m-2-3);
 \end{tikzpicture}
\end{equation}
If $\tilde{C}_1f$ and $\tilde{C}_2f$ are two different liftings \eqref{CTilde}, we have $$c_Y\circ\tilde{C}_1f=c_Y\circ\tilde{C}_2f\in\op{Hom}_{\tt M}(\tilde{C}X,Y)\;,$$ so that $\tilde{C}_1f\simeq^{\ell}\tilde{C}_2f\,.$ Similarly, if $\tilde{F}_1\tilde{C}_1f$ and $\tilde{F}_2\tilde{C}_2f$ are two different liftings \eqref{FTilde}, it follows from $$[\tilde{F}_1\tilde{C}_1f\circ f_{\tilde{C}X}]=[\tilde{F}_2\tilde{C}_2f\circ f_{\tilde{C}X}]\in\op{Hom}_{\tt M}(\tilde{C}X,\tilde{F}\tilde{C}Y)/\simeq^{r}$$ that $$[\tilde{F}_1\tilde{C}_1f]=[\tilde{F}_2\tilde{C}_2f]\in\op{Hom}_{\tt M}(\tilde{F}\tilde{C}X,\tilde{F}\tilde{C}Y)/\simeq\;=\;\op{Hom}_{{\tt Ho}_{\op{K}}({\tt M})}(\cL X,\cL Y)\;.$$ It can be straightforwardly checked that $\cL$ respects identities and compositions and is therefore a well-defined functor. Moreover, it is clear from \eqref{CFTilde} that if $f$ is a weak equivalence its lift $\tilde{F}\tilde{C}f$ is a weak equivalence between fibrant-cofibrant objects, hence a homotopy equivalence, which implies that $\cL f$ is an isomorphism.\medskip

To prove (L1) we must show that every functor $\cF\in[{\tt M},{\tt D}]_W$ can be written up to a natural isomorphism as the composite of $\cL$ and a functor $\tilde{\cF}\in{\tt Fun}({\tt Ho}_{\op{K}}({\tt M}),{\tt D})\,.$ We define $\tilde{\cF}$ on $$X\in{\tt M}_{\op{fc}}\quad\text{by}\quad \tilde{\cF} X:=\cF X\in\tt D$$ and on $$[f:X\to Y]\in\op{Hom}_{\tt M}(X,Y)/\simeq\;\;\;(X,Y\in{\tt M}_{\op{fc}})\quad\text{by}\quad \tilde{\cF}[f]:=\cF f\in\op{Hom}_{\tt D}(\tilde{\cF} X,\tilde{\cF} Y)\;.$$ The image $\tilde{\cF}[f]$ is well defined. Indeed, if $f,g:X\to Y$ are homotopic, there is a cylinder object $\op{Cyl}(X)$ or, more precisely, a factorization \be\label{WDI}w\circ i_1=\id_X\quad\text{and}\quad w\circ i_2=\id_X\,,\ee where $w$ is a weak equivalence and $i$ a cofibration (see above) and a morphism $H:\op{Cyl}(X)\to Y$ such that \be\label{WDII}H\circ i_1=f\;\text{and}\;H\circ i_2=g\;.\ee Applying $\cF\in[{\tt M,{\tt D}}]_W$ to the equalities in \eqref{WDI} we see that $\cF i_1=\cF i_2$ as $\cF w$ is an isomorphism, and applying it to the equalities in \eqref{WDII} we get that $\cF f=\cF g\,,$ so that $\tilde{\cF}$ is well defined on $[f]\,.$ Further $\tilde{\cF}$ respects identities and compositions since $\cF$ does and is therefore a functor $\tilde{\cF}\in\tt Fun(Ho_{\op{K}}(M),D)\,$. We now want to find a family indexed by $X\in{\tt M}$ of $\tt D$-isomorphisms $$\zh_X:\cF X\stackrel{\cong}{\to}\tilde{\cF}(\cL X)=\cF(\tilde{F}\tilde{C}X)\;,$$ such that for every $\tt M$-morphism $f:X\to Y$ we have $$\zh_Y\circ\cF f=\tilde{\cF}(\cL f)\circ \zh_X=\cF(\tilde{F}\tilde{C}f)\circ\zh_X\;.$$ Since if we apply $\cF$ to \eqref{CFTilde} we get the commutative diagram

\begin{equation}\label{CalFCFTilde} \begin{tikzpicture}
 \matrix (m) [matrix of math nodes, row sep=3em, column sep=3em]
   {  \cF X  & \cF(\tilde{C}X) & \cF(\tilde{F}\tilde{C}X)  \\
      \cF Y & \cF(\tilde{C}Y) & \cF(\tilde{F}\tilde{C}Y)  \\ };
 \path[->]
 (m-1-1) edge [->] node[above] {$(\cF c_X)^{-1}$} (m-1-2)
 (m-1-2) edge [->] node[auto] {$\cF(f_{\tilde{C}X})$}(m-1-3)
 (m-2-1) edge [->] node[below] {$(\cF c_Y)^{-1}$} (m-2-2)
 (m-2-2) edge [->] node[below] {$\cF(f_{\tilde{C}Y})$}(m-2-3)
 (m-1-1) edge [->] node[auto] {$\cF f$} (m-2-1)
 (m-1-2) edge [->] node[auto] {$\cF(\tilde{C}f)$} (m-2-2)
 (m-1-3) edge [->] node[auto] {$\cF(\tilde{F}\tilde{C}f)$} (m-2-3);
 \end{tikzpicture}
\end{equation}
it suffices to define $\zh_X$ as the isomorphism in the first line of \eqref{CalFCFTilde}.\medskip

It remains to check the fully faithfulness condition (L2') or equivalently \eqref{FulFai}, i.e., that for every $\cG,\cH\in{\tt Fun}({\tt Ho}_{\op{K}}({\tt M}),{\tt D})$ and every $\xi:\cG\circ\cL\Rightarrow\cH\circ\cL$ there is a unique $\zz:\cG\Rightarrow\cH$ such that $\zz\star\cL=\xi\,.$\medskip

Recall that $c_X:\tilde{C}X\stackrel{\sim}{\twoheadrightarrow} X$ (resp., $f_X:X\stackrel{\sim}{\rightarrowtail}\tilde{F}X$) is $\id_X$ if $X\in{\tt M}$ is cofibrant (resp., fibrant). In view of \eqref{CFTilde} we have therefore $\tilde{F}\tilde{C}f\simeq f$ if the source and target of $f$ are fibrant-cofibrant.\medskip

Hence if $\zz$ exists its components $\zz_X$ ($X\in{\tt M}_{\op{fc}}$) are necessarily given by $\zz_X=\zz_{\cL X}=\xi_X\,,$ so that $\zz$ is unique.\medskip

Conversely, if we set $\zz_X=\xi_X$ ($X\in{\tt M}_{\op{fc}}$) we get a family of $\tt D$-morphisms $\zz_X:\cG X\to \cH X\,.$ Since the transformation $\xi$ is natural, i.e., satisfies the obvious commutation condition for all $f\in\op{Hom}_{\tt M}(X,Y)$ ($X,Y\in{\tt M}$), the transformation $\zz$ is natural as well, i.e., satisfies this condition for all $[f]\in\op{Hom}_{{\tt Ho}_{\op{K}}({\tt M})}(X,Y)$ ($X,Y\in{\tt M}_{\op{fc}}$): $$\cH [f]\circ\zz_X=\cH(\cL f)\circ \xi_X=\xi_Y\circ\cG(\cL f)=\zz_Y\circ \cG[f]\,.$$ Further, by definition, we have $\zz_{\cL X}=\xi_{\cL X}$ ($X\in{\tt M}$) and get the needed result $\zz_{\cL X}=\xi_X$ if $X$ is fibrant and cofibrant. To conclude if $X$ is not necessarily fibrant-cofibrant we must prove that $\xi_{\cL X}=\xi_X$ for every $X\in{\tt M}\,.$ Using the naturality of $\xi$ for $f=f_{\tilde{C}X}\,,$ we get $$\cH(\cL f_{\tilde{C}X})\circ\xi_{\tilde{C}X}=\xi_{\cL X}\circ\cG(\cL f_{\tilde{C}X})$$ and using it for $f=c_X\,,$ we find $$\cH(\cL c_{X})\circ\xi_{\tilde{C}X}=\xi_{X}\circ\cG(\cL c_{X})\;,$$ where the images by $\cG\circ\cL$ and $\cH\circ\cL$ are isomorphisms. Therefore \be\label{NC1}\xi_{\cL X}=\cH(\cL f_{\tilde{C}X})\circ(\cH(\cL c_X))^{-1}\circ\xi_X\circ\cG(\cL c_X)\circ(\cG(\cL f_{\tilde{C}X}))^{-1}\;.\ee For $f:X\to Y$ equals $f_{\tilde{C}X}:\tilde{C}X\to \tilde{F}\tilde{C}X$ the lifting diagram \eqref{CFTilde} reads

$$\begin{tikzpicture}
 \matrix (m) [matrix of math nodes, row sep=3em, column sep=3em]
   {  \tilde{C}X  & \tilde{C}X & \tilde{F}\tilde{C}X  \\
      \tilde{F}\tilde{C}X & \tilde{F}\tilde{C}X & \tilde{F}\tilde{C}X  \\ };
 \path[->]
 (m-1-2) edge [->>] node[above] {\small{$\;\;{}_{\widetilde{}}\,\;\id_{\tilde{C}X}$}} (m-1-1)
 (m-1-2) edge [>->] node[auto] {\small{$\;\;{}_{\widetilde{}}\,\;f_{\tilde{C}X}$}}(m-1-3)
 (m-2-2) edge [->>] node[auto] {\small{$\;\;{}_{\widetilde{}}\,\;\id_{\tilde{F}\tilde{C}X}$}} (m-2-1)
 (m-2-2) edge [>->] node[below] {\small{$\;\;{}_{\widetilde{}}\,\;\id_{\tilde{F}\tilde{C}X}$}}(m-2-3)
 (m-1-1) edge [->] node[auto] {$f_{\tilde{C}X}$} (m-2-1)
 (m-1-2) edge [->] node[auto] {$\tilde{C}f_{\tilde{C}X}$} (m-2-2)
 (m-1-3) edge [->] node[auto] {$\tilde{F}\tilde{C}f_{\tilde{C}X}$} (m-2-3);
 \end{tikzpicture}
$$
so that \be\label{NC2}\cL f_{\tilde{C}X}=[\tilde{F}\tilde{C}f_{\tilde{C}X}]=[\id_{\tilde{F}\tilde{C}X}]=[\id_{\cL X}]\;.\ee Similarly, for $f:X\to Y$ equals $c_X:\tilde{C}X\to X$ the lifting diagram \eqref{CFTilde} reads

$$\begin{tikzpicture}
 \matrix (m) [matrix of math nodes, row sep=3em, column sep=3em]
   {  \tilde{C}X  & \tilde{C}X & \tilde{F}\tilde{C}X  \\
      \stackrel{}{X} & \tilde{C}X & \tilde{F}\tilde{C}X  \\ };
 \path[->]
 (m-1-2) edge [->>] node[above] {\small{$\;\;{}_{\widetilde{}}\,\;\id_{\tilde{C}X}$}} (m-1-1)
 (m-1-2) edge [>->] node[auto] {\small{$\;\;{}_{\widetilde{}}\,\;f_{\tilde{C}X}$}}(m-1-3)
 (m-2-2) edge [->>] node[auto] {\small{$\;\;{}_{\widetilde{}}\,\;c_X$}} (m-2-1)
 (m-2-2) edge [>->] node[below] {\small{$\;\;{}_{\widetilde{}}\,\;f_{\tilde{C}X}$}}(m-2-3)
 (m-1-1) edge [->] node[auto] {$c_X$} (m-2-1)
 (m-1-2) edge [->] node[auto] {$\tilde{C}c_X$} (m-2-2)
 (m-1-3) edge [->] node[auto] {$\tilde{F}\tilde{C}c_X$} (m-2-3);
 \end{tikzpicture}
$$
so that \be\label{NC3}\cL c_{X}=[\tilde{F}\tilde{C}c_X]=[\id_{\tilde{F}\tilde{C}X}]=[\id_{\cL X}]\;.\ee Combining \eqref{NC1}-\eqref{NC3}, we get $\xi_{\cL X}=\xi_X\,,$ which completes the proof.
\end{proof}

The next result is important for one of the comparison theorems of Subsection \ref{SSDerFunComp}.

\begin{cor}\label{BasicCor}
If $i:{\tt M}_{\op{c}}\hookrightarrow\tt M$ is the canonical inclusion functor, the pair $({\tt Ho}_{\op{K}}({\tt M}),\cL_{\tt M}\circ i)$ is a weak localization of $\,{\tt M}_{\op{c}}$ at $W\,.$ A similar result holds for ${\tt M}_{\op{f}}\,.$
\end{cor}

\begin{proof}
In the following we write $\cL$ instead of $\cL_{\tt M}\,.$ It is clear that $\cL\circ i\in\tt Fun(M_{\op{c}},Ho_{\op{K}}(M))$ sends weak equivalences to isomorphisms. We must show that for every functor $\cF\in[{\tt M}_{\op{c}},{\tt D}]_W$ there is a functor $\tilde{\cF}\in\tt Fun(Ho_{\op{K}}(M),D)$ and a natural isomorphism $\zh:\cF\stackrel{\cong}{\Rightarrow}\tilde{\cF}\circ\cL\circ i\,.$ We define $\tilde{\cF}$ as in the proof of Theorem \ref{KanWeak}, but before applying $\cF$ to the equalities in \eqref{WDI} and \eqref{WDII}, we have to check that the $\tt M$-morphisms $i_1, i_2: X\to\op{Cyl}(X)\,,$ $w:\op{Cyl}(X)\to X$ and $H:\op{Cyl}(X)\to Y$ have a cofibrant source and target, i.e., that $\op{Cyl}(X)$ is a cofibrant object. However, since cofibrations are closed under pushouts \cite{DS} and compositions, the coproduct of cofibrant objects is cofibrant and so is the cylinder of a cofibrant object. The remainder of the proof of Theorem \ref{KanWeak} goes through without difficulty.
\end{proof}

Notice that we denoted the localization functor by $\cL$ instead of $\cL_{\tilde{F}\tilde{C}}$ since it is essentially unique. Indeed, if $(\tilde{F}',\tilde{C}')$ is of the same type as $(\tilde{F},\tilde{C})$ and if we denote $\cL'$ the induced localization functor, the pairs $({\tt Ho_{\op{K}}(M)},\cL)$ and $({\tt Ho_{\op{K}}(M)},\cL')$ are both presentations of $({\tt M}[W^{-1}],L)\,.$ Using the unique isomorphisms \eqref{LInd2} and \eqref{LInd3} to identify the {\small LHS} and the {\small RHS} in \eqref{LInd2} and in \eqref{LInd3}, we conclude that $\tilde{\cL}:\tt Ho_{\op{K}}(M)\to Ho_{\op{K}}(M)$ is an automorphism of the homotopy category $\tt Ho_{\op{K}}(M)\,.$ As the pair made of the categorical automorphism $\tilde{\cL}$ and the natural isomorphism $\zh$ is unique, we can in view of \eqref{LInd1} identify $\cL$ and $\cL'\,.$\medskip

From what we said above it follows that the Kan homotopy category is characterized up to equivalence of categories by the {\it faint universal property} of Definition \ref{Faint}.

\subsubsection{Strong homotopy category}

\begin{defi} The {\bf Quillen homotopy category} or just the homotopy category ${\tt Ho(M)}$ of a model category $\tt M$ is the strong localization ${\tt M}[[W^{-1}]]$ of $\tt M$ at its class $W$ of weak equivalences.
\end{defi}

We know that $\tt Ho(M)$ does not have to exist, but is unique up to a unique isomorphism if it does.

\begin{theo}
The Quillen homotopy category $\,{\tt Ho}({\tt M})$ of a model category $\,\tt M$ does exist. The objects of $\,{\tt Ho}({\tt M})$ are the objects of $\,\tt M$ and its morphisms from $X$ to $Y$ are defined as $$\Hom_{{\tt Ho}({\tt M})}(X,Y):=\Hom_{\tt M}(\tilde{F}\tilde{C}X,\tilde{F}\tilde{C}Y)/\simeq\;\;,$$ where $\tilde{F}$ refers to a local fibrant C-replacement and $\tilde{C}$ to a local cofibrant F-replacement, or, equivalently, as $$\Hom_{{\tt Ho}({\tt M})}(X,Y):=\Hom_{\tt M}(RQX,RQY)/\simeq\;\;,$$ where $R$ and $Q$ are the fibrant and cofibrant replacement functors that are defined by a functorial factorization system. The localization functor $\zg:\tt M\to Ho(M)$ is defined on objects $X$ by $\zg X:=X$ and on morphisms $f:X\to Y$ by $\zg f:=[\tilde{F}\tilde{C}f]$, or, equivalently, by $\zg f:=[RQf]\,.$
\end{theo}

\begin{proof}
See \cite{Ho99}, \cite{Hir}.
\end{proof}

Actually $\zg(f)$ is an isomorphism if {\it and only if} $f$ is a weak equivalence \cite{Ho99}, \cite{Hir}.\medskip

There is a second description \cite{GZ,Ho99} of $\tt Ho(M)\,$. It starts from the free category ${\tt F}({\tt M},W)$ on $\tt M$ and $W$, whose objects are the objects of $\tt M$ and whose morphisms from $X$ to $Y$ are the zigzags from $X$ to $Y$, i.e., are the finite strings of morphisms of $\tt M$ and formal reversals $w^{-1}:V\rightsquigarrow U$ of weak equivalences $w:U\to V$ in $W$ that start at $X$ and arrive at $Y\,.$ Composition in ${\tt F}({\tt M},W)$ is concatenation and the identity at $X\in\tt M$ is the empty string $1_X$ at $X\,.$ From this free category ${\tt F}({\tt M},W)$ one then gets the homotopy category $\tt Ho(M)$ by identifying:
\begin{enumerate}
  \item the identity string $1_X:X\to X$ with the identity map (string) $\id_X:X\to X$ at $X$ in $\tt M$ (identification of identities),
  \item for any composable $\tt M$-maps $f:X\to Y$ and $g:Y\to Z\,$, the concatenation string $f,g:X\to Y\to Z$ with the composite (string) $g\circ f:X\to Z$ in $\tt M$ (identification of composites),
  \item for any weak equivalence $w:X\to Y$ of $\tt M\,$, the concatenation string $w,w^{-1}:X\to Y\rightsquigarrow X$ with the identity $\id_X:X\to X\,,$
  \item for any weak equivalence $w:X\to Y$ of $\tt M\,$, the concatenation string $w^{-1},w:Y\rightsquigarrow X\to Y$ with the identity $\id_Y:Y\to Y\,.$
\end{enumerate}
The resulting quotient category is the homotopy category $\tt Ho(M)\,.$ We denote the class of a string $S$ by $[S]\,.$\medskip

The localization functor $\zg$ is defined by $\zg X=X$ and $\zg(f:X\to Y)=[f:X\to Y]\in\Hom_{\tt Ho(M)}(X,Y)\,.$ As $\tt Ho(M)$ is a category, morphisms, i.e., classes of strings, can be composed and, of course, the composite of two composable classes is the class of the concatenation; moreover, there is an identity class at $X\in\tt Ho(M)$ which is obviously the class $[\id_X]=[1_X]\,.$ It follows that $\zg$ respects identities and composition: $$\zg(\id_X:X\to X)=[1_X]=[1_{\zg X}]\;$$ and $$\zg(g\circ f)=[g\circ f]=[X\to Y\to Z]=[g]\circ [f]=\zg g\circ \zg f\;.$$ The functor $\zg:\tt M\to Ho(M)$ sends weak equivalences $w$ to isomorphisms, i.e., the class $\zg w=[w]$ is an isomorphism, i.e., it has an inverse class $[w^{-1}]\,$; indeed \be\label{InvHoW}[w^{-1}]\circ[w]=[w,w^{-1}]=[\id_X]=[1_X]\;,\ee and similarly the other way round.\medskip

It is easy to check that if $w:X\to Y$ and $\zn:Y\to Z$ are weak equivalences, we have $[\zn^{-1},w^{-1}]=[(\zn\circ w)^{-1}]\,,$ so that any morphism of $\tt Ho(M)$ is an alternation $[\to\; \rightsquigarrow\; \to\; \rightsquigarrow ...]$ or $[\rightsquigarrow\; \to\; \rightsquigarrow\; \to ...]\,$.\medskip

Finally, if $({\tt D},F)$ is any category $\tt D$ together with a functor $F:\tt M\to D$ that sends weak equivalences to isomorphisms, and if there exists a functor ${\tt Ho}(F):{\tt Ho(M)}\to {\tt D}$ such that ${\tt Ho}(F)\circ\zg = F\,,$ we have necessarily $${\tt Ho}(F)(X)=F(X)\quad\text{and}\quad{\tt Ho}(F)[f:X\to Y]=F(f:X\to Y)\;.$$ Further, since for a weak equivalence $w:X\to Y$ we get $${\tt Ho}(F)[w^{-1}]\circ {\tt Ho}(F)[w]={\tt Ho}(F)[w,w^{-1}]={\tt Ho}(F)[\id_X]=\id_{F(X)}\;,$$ we also have necessarily $${\tt Ho}(F)[w^{-1}]=(F(w))^{-1}\;.$$ Conversely, when setting ${\tt Ho}(F)(X)=F(X)\,,$ ${\tt Ho}(F)(f:X\to Y)=F(f:X\to Y)\,$ and ${\tt Ho}(F)(w^{-1}:Y\rightsquigarrow X)=(F(w))^{-1}\,,$ we obtain a functor ${\tt Ho}(F):{\tt F}({\tt M},W)\to\tt \tt D$ that descends to the quotient category $\tt Ho(M)\,.$ Further, pre-composing ${\tt Ho}(F):{\tt Ho(M)}\to \tt D$ with $\zg\,,$ we get $F$.

\begin{rem}\label{ZigSubMod} The zigzag construction of $({\tt Ho(M)},\zg_{\tt M})$ as the strong localization $({\tt M}[[W^{-1}]],$ $L_{\tt M})$ is also valid for subcategories $\,\tt S$ of a model category $\,\tt M$ with $W$ being the $\tt S$-morphisms that are weak equivalences as $\tt M$-morphisms \cite{Ho99}. As elsewhere in the literature, we will use in this paper the notation $({\tt Ho(S)},\zg_{\tt S})$ although the correct notation is $({\tt S}[[W^{-1}]], L_{\tt S})\,.$\end{rem}

\subsubsection{Comparison theorem}

\begin{theo}\label{StrLocWeaLoc} The Quillen homotopy category ${\tt Ho(M)}$ of a model category $\tt M$ is a strict localization ${\tt M}[[W^{-1}]]$ of $\,\tt M$ at $W\,$.\end{theo}

Propositions \ref{SW}, \ref{WFSS} and \ref{UEUI} show that the following comparison result holds:

\begin{cor}
The Quillen homotopy category ${\tt Ho(M)}$ of a model category $\tt M$ is a weak localization ${\tt M}[W^{-1}]$ of $\,\tt M$ at $W$ and is therefore equivalent to the Kan homotopy category ${\tt Ho}_{\op{K}}({\tt M})\,.$
\end{cor}

We thus recover the well-known equivalence of $\tt Ho(M)$ and ${\tt Ho}_{\op{K}}({\tt M})\,.$ Conversely ${\tt Ho}_{\op{K}}({\tt M})$ is a weak localization ${\tt M}[W^{-1}]$ but it is not the strong localization ${\tt Ho(M)}={\tt M}[[W^{-1}]]$: it is a different category.

\begin{lem}\label{NatTraMHoM}
Let $F,G\in\tt Fun(Ho(M),D)\,.$ A family $\zy_{X}:F(X)\to G(X)$ of $\,\tt D$-maps indexed by the objects $X$ of $\,\tt M$ is a natural transformation $\zy:F\Rightarrow G$ if and only if it is a natural transformation $\zy:F\circ\zg_{\tt M}\Rightarrow G\circ\zg_{\tt M}\,.$
\end{lem}

\begin{proof}
The naturality condition for $\zy:F\circ\zg_{\tt M}\Rightarrow G\circ\zg_{\tt M}$ requires the obvious square to commute for all $\tt M$-maps $f:X\to Y\,,$ i.e., it reads \be\label{NCA}G[f]\circ\zy_X=\zy_Y\circ F[f]\;.\ee The naturality condition for $\zy:F\Rightarrow G$ requires the square to commute for all $\tt Ho(M)$-maps $[S]:X\to Y\,,$ i.e., it reads \be\label{NCB}G[S]\circ\zy_X=\zy_Y\circ F[S]\;,\ee where $[S]$ is an alternation of $\tt M$-maps $f$ and reversals $w^{-1}$ of weak equivalences $w\in W\,.$ Of course \eqref{NCB} implies \eqref{NCA}, but the opposite is also true. Indeed, assume \eqref{NCA} and consider the case $$[S]=[X\stackrel{f}{\to}Z\stackrel{w^{-1}}{\rightsquigarrow}Y]=[w^{-1}]\circ[f]\;.$$ In view of \eqref{InvHoW} we have $F[w^{-1}]=(F[w])^{-1}$ and similarly for $G\,,$ so that \eqref{NCA} gives \be\label{NCAW}\zy_Y\circ F[w^{-1}]=G[w^{-1}]\circ\zy_Z\;.\ee From \eqref{NCA} and \eqref{NCAW} it follows that $$G[S]\circ\zy_X=G[w^{-1}]\circ G[f]\circ\zy_X=G[w^{-1}]\circ\zy_Z\circ F[f]=\zy_Y\circ F[w^{-1}]\circ F[f]=\zy_Y\circ F[S]\;.$$
\end{proof}

\begin{proof}[Proof of Theorem \ref{StrLocWeaLoc}]
We have to prove that (L2') or equivalently \eqref{FulFai} holds. Hence let $F,G\in\tt Fun(Ho(M),D)$ and let $\zy:F\circ\zg_{\tt M}\Rightarrow G\circ\zg_{\tt M}\,$. If $\zz:F\Rightarrow G$ such that $\zz\star\zg_{\tt M}=\zy$ exists, its components are necessarily $\zz_X=\zy_X$ ($X\in{\tt M}$), so that $\zz$ is unique. Conversely, in view of Lemma \ref{NatTraMHoM} the family $\zz_X=\zy_X:F(X)\to G(X)$ of $\tt D$-maps indexed by $X\in\tt M$ is a natural transformation $\zz:F\Rightarrow G$ such that $\zz\star\zg_{\tt M}=\zy\,.$
\end{proof}

\section{Kan, faintly universal and strongly universal derived functors}\label{DerFuns}

\subsection{Left and right adjoint functors of the pullback by a functor}

If $P$ is a functor $P\in{\tt Fun(C,C')}$ and $\tt D$ is a category, the pre-composition $-\circ P$ is a functor $$P^*:{\tt Fun}({\tt C}',{\tt D})\to{\tt Fun(C,D)}\;,$$ whose value $P^*\zh'$ at a natural transformation $\zh'\in\op{Hom}_{\tt Fun(C',D)}(F',G')$ is the composition or whiskering $$\zh'\star P\in\op{Hom}_{\tt Fun(C,D)}(P^*(F'),P^*(G'))\;.$$ If $P^*$ has a left adjoint $P_{\,!}$ (resp., a right adjoint $P_*$), this adjoint is called {\it left Kan extension operation along} $P$ (resp., {\it right Kan extension operation along} $P$).\medskip

We will focus mainly on right extensions; left extensions are dual. The right adjoint $P_*$ exists if and only if $P^*$ is a left adjoint functor.\medskip

The universal morphism definition of a left adjoint functor $\cF:\tt F\to E$ constructs both the right adjoint $\cG$ and the counit $\ze:\cF\circ \cG\Rightarrow \id_{\tt E}$ of the adjunction: a functor $\cF:\tt F\to E$ is a left adjoint if and only if for any object $Y\in\tt E$ there is an object $\cG Y\in\tt F\,$ and a morphism $\ze_Y:\cF(\cG Y)\to Y$ such that $(\cG Y,\ze_Y)$ is universal, i.e., for any object $X\in\tt F$ and any morphism $\zz:\cF X\to Y$ there exists a unique morphism $f:X\to\cG Y$ such that $\ze_Y\circ\cF(f)=\zz\,.$ In this case, there is a unique way to extend $\cG$ to morphisms if one wants that $\ze$ becomes a natural transformation.\medskip

From this definition it follows that the right Kan extension operation $P_*$ along $P$ exists if and only if $P^*$ is a left adjoint if and only if for every $F\in\tt Fun(C,D)$ there is (a right extension) $P_*F\in\tt Fun(C',D)$ and a natural transformation $\ze_F: P^*(P_*F)\Rightarrow F\,,$ such that for every $F'\in\tt Fun(C',D)$ and every natural transformation $\zz:P^*F'\Rightarrow F$ there exists a unique natural transformation $\zy':F'\Rightarrow P_*F$ such that $\ze_F\circ P^*\zy'=\zz\,.$
\begin{equation} \begin{tikzpicture}
 \matrix (m) [matrix of math nodes, row sep=3em, column sep=3em]
   {  {\tt C}  & {\tt C}'  \\
       & {\tt D}  \\ };
 \path[->]
 (m-1-1) edge node[left] {\small{$F$}} (m-2-2)
 (m-1-1) edge node[above] {\small{$P$}} (m-1-2)
 (m-1-2) edge [->, dashed] node[auto] {\small{$P_*F$}} (m-2-2);
 \draw[-latex] node[auto]{$\;\;\;\;\;{\ze_F}$} node[above,rotate=45]{$\;\;\;\;\;\;\Leftarrow$};
\end{tikzpicture}.
\end{equation}
As usual, a universal pair $(P_*F,\ze_F)$ need not exist, but if it does, it is unique up to unique natural isomorphism. The universal pair $(P_*F,\ze_F)$ can exist locally, i.e., for specific objects $F\in\tt Fun(C,D)\,,$ without $P_*$ existing globally as right adjoint functor of $P^*\,.$ If $(P_*F,\ze_F)$ exists for a given $F$ we refer to $P_*F$ as the {\it right Kan extension of $F$ along $P\,.$}\medskip

For instance, let $\tt D$ be a category, let $\tt C$ be a small category and $\tt C'$ the terminal category $[\diamond]$ (with one object $\diamond$ and one morphism $\id_\diamond$). The unique functor $P\in\tt Fun(C,[\diamond])$ to the terminal category is the constant functor at $\diamond\,$, there is a canonical isomorphism $\,\tt Fun([\diamond],D)\cong D\,,$ and the pre-composition $P^*:{\tt D}\to{\tt Fun(C,D)}$ is the constant functor $$-^*:{\tt D}\to{\tt Fun(C,D)}\;,$$ which sends every object $d\in\tt D$ to the constant functor $d^*$ at $d$ and every $\tt D$-morphism $g:d\to d'$ to the constant natural transformation $g^*:d^*\Rightarrow d'^*$ with components $g^*_c=g$ ($c\in\tt C$). If $\tt D$ has all limits indexed by $\tt C\,,$ the limit $$\op{Lim}:\tt Fun(C,D)\to D$$ is a functor that is right adjoint to $-^*\,,$ i.e., $\op{Lim}$ is the right Kan extension operation along $P$. Indeed, the latter exactly means that for every $F\in\tt Fun(C,D)$ there is a functor $\op{Lim}F\in\tt Fun([\diamond],D)$ such that the diagram
\begin{equation} \begin{tikzpicture}
 \matrix (m) [matrix of math nodes, row sep=3em, column sep=3em]
   {  {\tt C}  & \text{$[\diamond]$}  \\
       & {\tt D}  \\ };
 \path[->]
 (m-1-1) edge node[left] {\small{$F$}} (m-2-2)
 (m-1-1) edge node[above] {\small{$P$}} (m-1-2)
 (m-1-2) edge [->, dashed] node[auto] {\small{$\op{Lim}F$}} (m-2-2);
 \draw[-latex] node[auto]{$\;\;\;\;\;{\ze}$} node[above,rotate=45]{$\;\;\;\;\;\;\Leftarrow$};
\end{tikzpicture}
\end{equation}
commutes up to a natural transformation $\ze:(\op{Lim}F)\circ P\Rightarrow F\,,$ and that the pair $(\op{Lim}F,\ze)$ is universal; in other words, the functor $\op{Lim}$ is the right Kan extension operation along P means precisely that for every $F\in\tt Fun(C,D)$ there is an object $\op{Lim}F\in\tt D$ and a family indexed by $c\in\tt C$ of $\tt D$-maps $\ze_c:\op{Lim}F\to F(c)$ such that for each $\tt C$-morphism $f:c\to c'$ we have $\ze_{c'}=F(f)\circ\ze_c\,,$ and that the pair $(\op{Lim}F, (\ze_c)_{c\in\tt C})$ is universal:
\begin{equation} \begin{tikzpicture}
 \matrix (m) [matrix of math nodes, row sep=1.5em, column sep=1.5em]
   {  &\text{$\op{Lim}F$}& \\ &&\\ \text{$F(c)$}
       & & \text{$F(c')$} \\ };
 \path[->]
 (m-1-2) edge node[above] {\small{$\ze_c\;\;$}} (m-3-1)
 (m-1-2) edge node[auto] {\small{$\ze_{c'}$}} (m-3-3)
 (m-3-1) edge node[auto] {\small{$F(f)$}} (m-3-3);
\end{tikzpicture}.
\end{equation}

\subsection{K and F derived functors on a category with a distinguished family of maps}\label{KFCatW}

Roughly speaking the derived functor of a functor from $\tt C$ to $\tt D$ is a functor from `the localization' of $\tt C$ to $\tt D$.\medskip

Let $\tt C$ be a category with a family $W$ of maps whose faint localization $L\in{\tt Fun(C,C}[W^{-1}])$ exists. If $F\in\tt Fun(C,D)\,,$ we have two possibilities to get a (left) {\it derived functor} $$\mathbf{L}F\in{\tt Fun(C}[W^{-1}],{\tt D})\;.$$ First we can use the right Kan extension of $F$ along $L$ and set $$\mathbf{L}^{\op{K}}F:=L_*F\in{\tt Fun(C}[W^{-1}],{\tt D})\;,$$ {\bf provided this extension exists}, i.e., the universal pair $(L_*F,\ze)\,,$ where $$\ze:L_*F\circ L\Rightarrow F\;,$$ exists. Of course if it does, it is unique (among universal pairs) up to unique natural isomorphism. Secondly, {\bf if $F$ sends maps in $W$ to isomorphisms}, there {\bf exists} in view of the faint universality of $({\tt C}[W^{-1}],L)$ a universal pair ($\tilde{F},\zi$), where $$\zi:\tilde{F}\circ L\stackrel{\cong}{\Rightarrow} F\;,$$ and we can set $$\mathbf{L}^{\op{F}}F:=\tilde{F}\in{\tt Fun(C}[W^{-1}],{\tt D})\;.$$ This pair is by definition unique (among similar pairs) up to unique natural isomorphism. The derived functors $\mathbf{L}^{\op{K}}F$ and $\mathbf{L}^{\op{F}}F$ are different, i.e., the pair $(\mathbf{L}^{\op{K}}F,\ze)$ is generally not a universal pair in the sense of the faint universal property and $(\mathbf{L}^{\op{F}}F,\zi)$ is usually not a right Kan extension, although this can be the case in specific situations (see Theorem \ref{FundamentalA}).\medskip

If $\tt D$ is also a category with a distinguished family of maps, $V$ say, whose faint localization $L_{\tt D}\in{\tt Fun(D,D}[V^{-1}])$ exists, one mostly considers {\it total derived functors} $$\mathbb{L}F\in{\tt Fun(C}[W^{-1}],{\tt D}[V^{-1}])$$ of functors $F\in\tt Fun(C,D)\,.$ Since $L_{\tt D}\circ F\in{\tt Fun(C,D}[V^{-1}])\,,$ it suffices to set $$\mathbb{L}^{\op{K}}F:=\mathbf{L}^{\op{K}}(L_{\tt D}\circ F)\;,$$ provided the {\small RHS} exists. If $F$ sends maps in $W$ to maps in $V\,,$ we can set as well $$\mathbb{L}^{\op{F}}F:=\mathbf{L}^{\op{F}}(L_{\tt D}\circ F)\;,$$ where the {\small RHS} does exist.\medskip

Right derived functors $\mathbf{R}F$ and total right derived functors $\mathbb{R}F$ are defined dually.

\begin{rem}\label{LRFDF1}
\emph{Left and right derived functors differ by the direction of the natural transformation. Since this transformation is an isomorphism for {\small F} derived functors, the left and right {\small F} derived functors coincide so far. Later we will work with {\small F} derived functors that we will call left derived and others that we call right derived (see Remark \ref{LRFDF2}).}
\end{rem}

\subsection{K, F and S derived functors in model categories}

In this subsection we consider only total derived functors of functors between model categories.\medskip

Hence let $F\in\tt Fun(M,N)$ be a functor between model categories $\tt M, N\,.$ We denote the class of weak equivalences of $\,\tt M$ (resp., $\tt N$) by $W$ (resp., $V$). The localization $({\tt M}[W^{-1}],L_{\tt M})$ (resp., $({\tt N}[V^{-1}],L_{\tt N})$) exists and admits the equivalent categories ${\tt Ho(M)}\approx {\tt Ho_{\op{K}}(M)}$ with their localization functors $\zg_{\tt M}$ and $\cL_{\tt M}$ (resp., ${\tt Ho(N)}\approx {\tt Ho_{\op{K}}(N)}$ with $\zg_{\tt N}$ and $\cL_{\tt N}$) as presentations.\medskip

In model categories there are a number of possible approaches to total derived functors. For each of these types of derived functor, we give a precise definition, emphasizing in particular which localization and localization property we are using, we state existence and uniqueness results under certain conditions, and we highlight the type of `commutation' relation that the type of derived functor considered satisfies.

\subsubsection{K derived functors in model categories}

In this section we {\it mainly} use the {\bf presentation $({\tt Ho(M)},\zg_{\tt M})$ of $({\tt M}[W^{-1}],L_{\tt M})$}.\medskip

Just as in general categories whose faint localization at a distinguished family of morphisms exists, we choose the

\begin{defi} Let $\tt M$ and $\tt N$ be model categories and let $F\in\tt Fun(M,N)\,.$ The {\bf {\small K} total derived functor} $\mathbb{L}^{\op{K}}F\in\tt Fun(Ho(M),Ho(N))$ is the right Kan extension $$\mathbb{L}^{\op{K}}F:=\mathbf{L}^{\op{K}}(\zg_{\tt N}\circ F)=(\zg_{\tt M})_*(\zg_{\tt N}\circ F)\;$$ of $\zg_{\tt N}\circ F$ along $\zg_{\tt M}$ provided this extension exists.
\end{defi}

The following existence and uniqueness result holds:

\begin{prop}\label{ExistKDerFun}
If $F\in\tt Fun(M,N)$ {\bf sends weak equivalences between cofibrant objects to weak equivalences}, the total left derived functor $$\mathbb{L}^{\op{K}}F=\mathbf{L}^{\op{K}}(\zg_{\tt N}\circ F)=(\zg_{\tt M})_*(\zg_{\tt N}\circ F)\in\tt Fun(Ho(M),Ho(N))$$ {\bf exists and is unique up to unique natural isomorphism}. More precisely (Subsection \ref{KFCatW}), the {\small K} derived functor of $F$ comes with a {\bf natural transformation} \be\label{U1}\ze:\mathbb{L}^{\op{K}}F\circ\zg_{\tt M}\Rightarrow \zg_{\tt N}\circ F\;,\ee the pair $(\mathbb{L}^{\op{K}}F,\ze)$ is universal and it is this universal pair that is unique.
\end{prop}

We will prove Proposition \ref{ExistKDerFun} in Subsection \ref{SSDerFunComp}.\medskip

The dual result holds for the total right derived functor $\mathbb{R}^{\op{K}}F:\tt Ho(M)\to Ho(N)$. In particular, if $F:{\tt M\to N}$ is a left Quillen functor, it respects trivial cofibrations, hence sends trivial cofibrations between cofibrant objects to weak equivalences, and in view of Brown's lemma sends all weak equivalences between cofibrant objects to weak equivalences, so that $\mathbb{L}^{\op{K}}F$ exists.

\begin{prop}
If $F:{\tt M\to N}:G$ is a Quillen adjunction, both functors $\mathbb{L}^{\op{K}}F:{\tt Ho(M)\rightleftarrows Ho(N)}:\mathbb{R}^{\op{K}}G$ exist and are adjoint functors.
\end{prop}

Recall that a Quillen adjunction (`morphism of model categories') $F:{\tt M\to N}:G$ is a Quillen equivalence (`equivalence of model categories') if and only if the $(X,Y)$-natural adjunction bijections $$\Hom_{\tt N}(FX,Y)\ni f\cong f_{\flat}\in\Hom_{\tt M}(X,GY)$$ respect weak equivalences if $X$ is cofibrant and $Y$ is fibrant, i.e., $f$ is a weak equivalence if and only if $f_\flat$ is a weak equivalence.

\begin{prop}If $F:{\tt M\to N}:G$ is a Quillen equivalence, then $\mathbb{L}^{\op{K}}F:{\tt Ho(M)\rightleftarrows Ho(N)}:\mathbb{R}^{\op{K}}G$ is an equivalence of categories.\end{prop}

A result similar to Proposition \ref{ExistKDerFun} holds for the {\bf presentation $({\tt Ho_{\op{K}}(M)},\cL_{\tt M})$ of $({\tt M}[W^{-1}],L_{\tt M})$}. If $i:{\tt M}_{\op{c}}\hookrightarrow\tt M$ is the canonical inclusion functor, we have:

\begin{prop}\label{ExistKDerFunK}
If $F\in\tt Fun(M,N)$ {\bf sends weak equivalences between cofibrant objects to weak equivalences}, the total left derived functor $$\mathbb{L}^{\op{K}}F=\mathbf{L}^{\op{K}}(\cL_{\tt N}\circ F\circ i)=(\cL_{\tt M}\circ i)_\star(\cL_{\tt N}\circ F\circ i)\in\tt Fun(Ho_{\op{K}}(M),Ho_{\op{K}}(N))$$ {\bf exists and is unique up to unique natural isomorphism}. More precisely (Subsection \ref{KFCatW}), the {\small K} derived functor of $F$ comes with a {\bf natural transformation} \be\label{U1}\ze:\mathbb{L}^{\op{K}}F\circ\cL_{\tt M}\circ i\Rightarrow \cL_{\tt N}\circ F\circ i\;,\ee the pair $(\mathbb{L}^{\op{K}}F,\ze)$ is universal and it is this universal pair that is unique.
\end{prop}

We will also prove Proposition \ref{ExistKDerFunK} in Subsection \ref{SSDerFunComp}.

\subsubsection{F derived functors in model categories}

We will use the {\bf presentation $({\tt Ho_{\op{K}}(M)},\cL_{\tt M})$ of the localization $({\tt M}[W^{-1}],L_{\tt M})$}.\medskip

Again, just as in general categories with distinguished families of morphisms, we can define the total derived functor $$\mathbb{L}F\in\tt Fun(Ho_{\op{K}}(M),Ho_{\op{K}}(N))$$ using the faint universal property of $\tt Ho_{\op{K}}(M)\,.$ More precisely, let $i:{\tt M}_{\op{c}}\hookrightarrow \tt M$ be as usual the canonical inclusion functor of the full subcategory of cofibrant objects. If $F\in\tt Fun(M,N)$ sends weak equivalences between cofibrant objects to weak equivalences, the composite of $F\circ i\in\tt Fun(M_{\op{c}},N)$ and $\cL_{\tt N}\in\tt Fun(N,Ho_{\op{K}}(N))\,$ sends weak equivalences to isomorphisms. Hence it factors up to a natural isomorphism through ${\tt M_{\op{c}}}[W^{-1}]$ thus leading to a functor $$\mathbf{L}^{\op{F}}(\cL_{\tt N}\circ F\circ i)\in{\tt Fun(M_{\op{c}}}[W^{-1}],\tt Ho_{\op{K}}(N))\;.$$ Let \be\label{NatIso0}\imath:\mathbf{L}^{\op{F}}(\cL_{\tt N}\circ F\circ i)\circ L_{\tt M_{\op{c}}}\stackrel{\cong}{\Rightarrow}\cL_{\tt N}\circ F\circ i\;\ee be this natural isomorphism. From Subsection \ref{KFCatW} we know that the pair $(\mathbf{L}^{\op{F}}(\cL_{\tt N}\circ F\circ i),\imath)$ is unique up to unique natural isomorphism, of course, provided the localization $({\tt M_{\op{c}}}[W^{-1}],L_{\tt M_{\op{c}}})$ exists.\medskip

However, in view of Corollary \ref{BasicCor} the {\bf pair $({\tt Ho_{\op{K}}(M)},\cL_{\tt M}\circ\, i)$ is a presentation of the localization $({\tt M_{\op{c}}}[W^{-1}],L_{\tt M_{\op{c}}})$} of the cofibration category $\tt M_{\op{c}}$ \cite{JL}. A similar result holds for $j:{\tt M_{\op{f}}\hookrightarrow M}\,.$ Many authors denote these presentations by $(\tt Ho_{\op{K}}(M_{\op{c}}),\cL_{M_{\op{c}}})$ and $(\tt Ho_{\op{K}}(M_{\op{f}}),\cL_{M_{\op{f}}})\,,$ respectively.\medskip

The next proposition follows from what we just said.

\begin{prop}\label{EUFDerFun} If $F\in\tt Fun(M,N)$ {\bf sends weak equivalences between cofibrant objects to weak equivalences}, the {\bf {\small F} total derived functor} $$\mathbb{L}^{\op{F}} F:=\mathbf{L}^{\op{F}}(\cL_{\tt N}\circ F\circ i)\in\tt Fun(Ho_{\op{K}}(M),Ho_{\op{K}}(N))$$ together with the {\bf natural isomorphism} \be\label{U2}\imath:\mathbb{L}^{\op{F}} F\circ \cL_{\tt M}\circ i\stackrel{\cong}{\Rightarrow}\cL_{\tt N}\circ F\circ i\;\ee are {\bf unique up to unique natural isomorphism}.
\end{prop}

Of course, if $F:{\tt M\rightleftarrows N}:G$ is a Quillen adjunction, then $F$ respects trivial cofibrations, hence satisfies the previous condition, so that $\mathbb{L}^{\op{F}}F:\tt Ho_{\op{K}}(M)\to Ho_{\op{K}}(N)$ exists. Dually, the total right derived functor $\mathbb{R}^{\op{F}}G:\tt Ho_{\op{K}}(N)\to Ho_{\op{K}}(M)$ exists:

\begin{prop} If $F:{\tt M\rightleftarrows N}:G$ is a Quillen adjunction, the functor $\mathbb{L}^{\op{F}}F$ is left adjoint to the functor $\mathbb{R}^{\op{F}}G\,.$ The derived functors are an equivalence of categories if the Quillen adjunction considered is a Quillen equivalence.\end{prop}

\begin{rem}\label{LRFDF2}
\emph{Notice that the left {\small F} derived functor $\mathbb{L}^{\op{F}}F$ is defined by means of the inclusion functor $i:{\tt M}_{\op{c}}\hookrightarrow\tt M\,,$ whereas the right {\small F} derived functor $\mathbb{R}^{\op{F}}G$ uses the inclusion functor $j:{\tt M}_{\op{f}}\hookrightarrow\tt M\,.$ Therefore the distinction between left and right is important here (see Remark \ref{LRFDF1}).}
\end{rem}

\subsubsection{S derived functors in model categories}\label{SSUDerFun}

In model categories we have still another option for defining total derived functors. Indeed, as mentioned before, the {\bf pair $({\tt Ho(M)},\zg_{\tt M})$} associated to any model category $\tt M$ is not only a presentation of the faint localization $({\tt M}[W_{\tt M}^{-1}],L_{\tt M})$ of $\tt M$ at its weak equivalences $W_{\tt M}\,,$ it {\bf is also the strong localization} $({\tt M}[[W_{\tt M}^{-1}]],L_{\tt M})\,.$ In view of Remark \ref{ZigSubMod}, the strong localization $({\tt S}[[W_{\tt S}^{-1}]],L_{\tt S})$ of a subcategory $\tt S$ of a model category $\tt M$ at the class $W_{\tt S}$ of $\tt S$-morphisms that belong to $W_{\tt M}$, is given by the zigzag construction used for the strong localization $({\tt Ho(M)},\zg_{\tt M})$ of $\tt M$ and it is denoted a bit abusively by $({\tt Ho(S)},\zg_{\tt S})\,.$ Thus we can fall back for instance on the strong localization $({\tt Ho(M_{\op{c}})},\zg_{\tt M_{\op{c}}})$ of $\tt M_{\op{c}}$ at the class $W_{\tt M_{\op{c}}}$ of morphisms in $W_{\tt M}$ that act between cofibrant objects. In the following we will denote both classes $W_{\tt M_{\op{c}}}$ and $W_{\tt M}$ by $W\,.$\medskip

Hence, if $F\in\tt Fun(M,N)$ sends weak equivalences between cofibrant objects to weak equivalences, then $\zg_{\tt N}\circ F\circ i\in \tt Fun(M_{\op{c}},Ho(N))$ sends weak equivalences to isomorphisms and therefore factors {\it uniquely} and {\it on the nose} through $\tt Ho(M_{\op{c}})\,,$ i.e., there exists a unique functor $$\op{Ho}(F):=\op{Ho}(\zg_{\tt N}\circ F\circ i)\in\tt Fun(Ho(M_{\op{c}}),Ho(N))\;,$$ such that \be\label{NatIso1}\op{Ho}(F)\circ\zg_{\tt M_{\op{c}}}=\zg_{\tt N}\circ F\circ i\;.\ee

There is an equivalence of categories $\tt Ho(M_{\op{c}})\approx Ho(M)\,.$ The functor from $\tt Ho(M_{\op{c}})$ to $\tt Ho(M)$ is the unique factorization $\op{Ho}(i):=\op{Ho}(\zg_{\tt M}\circ i)$ through $\tt Ho(M_{\op{c}})$ of $\zg_{\tt M}\circ i\in\tt Fun(M_{\op{c}},Ho(M))\,.$ It is well-known that the quasi-inverse of an equivalence of categories is unique up to isomorphism (the quasi-inverse of an adjoint equivalence of categories is unique up to unique isomorphism).

\begin{defi} If $F\in\tt Fun(M,N)$ {\bf sends weak equivalences between cofibrant objects to weak equivalences} and if $\,\cI\in\tt Fun(Ho(M), Ho(M_{\op{c}}))$ is a {\bf quasi-inverse} of $\,\op{Ho}(i)\,,$  we define the {\bf {\small S} total derived functor} of $F$ by $$\mathbb{L}^{\op{S}}_\cI F:=\op{Ho}(F)\circ \cI\in\tt Fun(Ho(M),Ho(N))\;.$$
\end{defi}

If $\cJ$ is another quasi-inverse, there is a natural isomorphism $\mathfrak{i}:\cI\stackrel{\cong}{\Rightarrow}\cJ$ and $\zk:=\op{Ho(F)}\,\star\,\mathfrak{i}\,$ is a natural isomorphism $\zk:\mathbb{L}^{\op{S}}_\cI F\stackrel{\cong}{\Rightarrow}\mathbb{L}^{\op{S}}_\cJ F\,.$ Hence:

\begin{prop} If $F\in\tt Fun(M,N)$ sends weak equivalences between cofibrant objects to weak equivalences, the total derived functor $$\mathbb{L}^{\op{S}} F\in\tt Fun(Ho(M),Ho(N))\;$$ {\bf exists and is essentially unique}, i.e., whatever quasi-inverse of $\op{Ho}(i)$ we choose to compute the derived functor we get a representative in the same isomorphism class. 
\end{prop}

\begin{rem}
\emph{Let us emphasize very clearly that $\mathbb{L}^{\op{S}} F$ is defined up to a natural isomorphism: $\mathbb{L}^{\op{S}}F$ stands for any of the $\mathbb{L}^{\op{S}}_\cI F\,,$ where $\cI$ is a quasi-inverse of $\op{Ho}(i)\,.$}
\end{rem}

For instance, every cofibrant F-replacement functor $Q:\tt M\to M_{\op{c}}$ induces a quasi-inverse of $\op{Ho}(i)\,.$ Indeed, as we interpret $Q$ here as a functor to $\tt M_{\op{c}}$ and not to $\tt M$ as before in this text, the natural weak equivalence $q:Q\stackrel{\sim}{\Rightarrow}\id_{\tt M}$ mentioned in Subsection \ref{Rep} can be written \be\label{NatIso2}q:i\circ Q\stackrel{\sim}{\Rightarrow}\id_{\tt M}\ee and it restricts to a natural weak equivalence $q\star i:Q\circ i\stackrel{\sim}{\Rightarrow}\id_{\tt M_{\op{c}}}.$ It is now easy to see that the unique factorization $\op{Ho}(Q):=\op{Ho}(\zg_{\tt M_{\op{c}}}\circ Q):\tt Ho(M)\to Ho(M_{\op{c}})$ is a quasi-inverse of $\op{Ho}(i)\,.$ Hence:

\begin{prop}\label{WCWQ} If $F\in\tt Fun(M,N)$ sends weak equivalences between cofibrant objects to weak equivalences and if $Q\in\tt Fun(M,M_{\op{c}})$ is a cofibrant F-replacement functor, the total derived functor of $F$ is given by \be\label{NatIso3}\mathbb{L}^{\op{S}}F\stackrel{\cong}{\Rightarrow}\mathbb{L}^{\op{S}}_Q F=\op{Ho}(F)\circ\op{Ho}(Q)\;.\ee Moreover, we have the {\bf equality} \be\label{U3}\mathbb{L}^{\op{S}}_QF\circ \zg_{\tt M}=\zg_{\tt N}\circ F\circ Q\;.\ee\end{prop}

Indeed, the unique factorization $\op{Ho}(Q)$ satisfies \be\label{NatIso4}\op{Ho}(Q)\circ\zg_{\tt M}=\zg_{\tt M_{\op{c}}}\circ Q\;\ee and the previous equality follows from \eqref{NatIso3}, \eqref{NatIso4} and \eqref{NatIso1}.\medskip

Let us mention that since the {\small RHS} of \eqref{U3} can be written $\zg_{\tt N}\circ F\circ i\circ Q\,,$ we get from \eqref{NatIso3}, \eqref{U3} and \eqref{NatIso2} a natural transformation \be\label{NU2}\mathbb{L}^{\op{S}}F\circ \zg_{\tt M}\Rightarrow\zg_{\tt N}\circ F\;.\ee A similar observation can be deduced in the case of $\op{F}$ derived functors from the whiskering of the natural isomorphism $\imath$ in \eqref{U2} with $Q\,$: $$\imath\star Q:\mathbb{L}^{\op{F}} F\circ \cL_{\tt M}\circ i\circ Q\stackrel{\cong}{\Rightarrow}\cL_{\tt N}\circ F\circ i\circ Q\;.$$ Indeed, in view of \eqref{NatIso2} we have a natural isomorphism $$\mathbb{L}^{\op{F}} F\circ \cL_{\tt M}\circ i\circ Q\stackrel{\cong}{\Rightarrow}\mathbb{L}^{\op{F}} F\circ \cL_{\tt M}\;$$ and a natural transformation $$\cL_{\tt N}\circ F\circ i\circ Q\Rightarrow \cL_{\tt N}\circ F\;,$$ and so get a natural transformation \be\label{NU3}\mathbb{L}^{\op{F}} F\circ \cL_{\tt M}\Rightarrow\cL_{\tt N}\circ F\;.\ee {\it The interesting equations are \eqref{U2} and \eqref{U3} which are stronger than the $\op{S}$ and $\op{F}$ counterparts \eqref{NU2} and \eqref{NU3} of the $\op{K}$ equation \eqref{U1}.}\medskip

{\it In order to compute total derived functors one usually pre-composes the original functor with a cofibrant F-replacement functor $Q\,$.}\medskip

Indeed, the functor $\zg_{\tt N}\circ F\circ Q\in \tt Fun(M,Ho(N))$ sends weak equivalences to isomorphisms, so that there exists a unique functor \be\label{ExistKDF1}\op{Ho}(F\circ Q):=\op{Ho}(\zg_{\tt N}\circ F\circ Q)\in\tt Fun(Ho(M),Ho(N))\;,\ee such that \be\label{ExistKDF2}\op{Ho}(F\circ Q)\circ\zg_{\tt M}=\zg_{\tt N}\circ F\circ Q\;.\ee In view of Equation \eqref{U3} in Proposition \ref{WCWQ} we get:

\begin{prop}\label{CompDerFunP1} Under the assumptions of Proposition \ref{WCWQ}, the total derived functor of $F$ is given by \be\label{CompDerFunE1}\mathbb{L}^{\op{S}}F\stackrel{\cong}{\Rightarrow}\mathbb{L}_Q^{\op{S}}F=\op{Ho}(F)\circ\op{Ho}(Q)=\op{Ho}(F\circ Q)\;.\ee \end{prop}

\begin{rem} The results on S {\it derived functors of Quillen functors} are the same as in the case of {\small K} and {\small F} derived functors.\end{rem}

\subsubsection{Comparison theorems}\label{SSDerFunComp}

Above we used the faint localization $({\tt Ho}_{\op{K}}({\tt M}),\cL_{\tt M})$ of $\tt M$ at $W$ and considered K and F derived functors, and we used the strong localization $({\tt Ho}({\tt M}),\zg_{\tt M})$ of $\tt M$ at $W$ and looked at K and S derived functors.\medskip

We will show that in both cases the two derived functors under consideration are equal, and begin with the following refinement of Proposition \ref{ExistKDerFunK}.

\begin{theo}\label{FundamentalA}
If $F\in\tt Fun(M,N)$ sends weak equivalences between cofibrant objects to weak equivalences, the Kan extension derived functor $$\mathbb{L}^{\op{K}}F=\mathbf{L}^{\op{K}}(\cL_{\tt N}\circ F\circ i)=(\cL_{\tt M}\circ i)_*(\cL_{\tt N}\circ F\circ i)\in\tt Fun(Ho_{\op{K}}(M),Ho_{\op{K}}(N))$$ exists and is given by $$\mathbb{L}^{\op{F}}F=\mathbf{L}^{\op{F}}(\cL_{\tt N}\circ F\circ i)\;,$$ where $i:{\tt M}_{\op{c}}\hookrightarrow \tt M$ is the canonical inclusion.
\end{theo}

The usually different K and F derived functors coincide here since in view of Corollary \ref{BasicCor} the pair $({\tt Ho}_{\op{K}}({\tt M}),\cL_{\tt M}\circ i)$ is not only a faint but a weak localization of $\,{\tt M}_{\op{c}}$ at $W$:

\begin{proof}
From Proposition \ref{EUFDerFun} it follows that $\mathbb{L}^{\op{F}}F\in\tt Fun(Ho_{\op{K}}(M),Ho_{\op{K}}(N))$ satisfies $$\imath:\mathbb{L}^{\op{F}} F\circ \cL_{\tt M}\circ i\stackrel{\cong}{\Rightarrow}\cL_{\tt N}\circ F\circ i\;.$$ Hence $\mathbb{L}^{\op{K}}F=\mathbb{L}^{\op{F}}F$ if for any functor $G\in\tt Fun(Ho_{\op{K}}(M),Ho_{\op{K}}(N))$ and any natural transformation $$\xi:G\circ\cL_{\tt M}\circ i\Rightarrow\cL_{\tt N}\circ F\circ i\;,$$ there is a unique natural transformation $\zz:G\Rightarrow \mathbb{L}^{\op{F}}F$ such that $$\imath\circ (\zz\star(\cL_{\tt M}\circ i))=\xi\;.$$ However, as $({\tt Ho}_{\op{K}}({\tt M}),\cL_{\tt M}\circ i)$ is a weak localization, the functor $-\circ\cL_{\tt M}\circ i$ is fully faithful, so that the natural transformation $$\imath^{-1}\circ\xi:G\circ\cL_{\tt M}\circ i\Rightarrow\mathbb{L}^{\op{F}}F\circ\cL_{\tt M}\circ i$$ reads $\zz\star(\cL_{\tt M}\circ i)$ for a unique $\zz:G\Rightarrow\mathbb{L}^{\op{F}}F\,.$
\end{proof}

We finally prove Proposition \ref{ExistKDerFun}. More precisely, we show that if $F\in\tt Fun(M,N)$ sends weak equivalences between cofibrant objects to weak equivalences, the Kan extension derived functor $$\mathbb{L}^{\op{K}}F\in\tt Fun(Ho(M),Ho(N))$$ exists and is given by $\mathbb{L}_Q^{\op{S}}F=\op{Ho}(F\circ Q)$ and by $\op{Ho}(F\circ\tilde{C})\,,$ where $Q$ and $\tilde{C}$ are defined as usual.

\begin{proof}[Proof of Proposition \ref{ExistKDerFun}]
Let $\tilde{C}$ be a {\it local cofibrant F-replacement} as in Subsection \ref{Rep} and recall (see proof of Theorem \ref{KanWeak}) that $\tilde{C}$ is an endofunctor of $\tt M$ up to left homotopy, in the sense that its value at an $\tt M$-morphism $f:X\to Y$ is well-defined only up to left homotopy and therefore it respects compositions and identities only up to left homotopy. Nevertheless $$\cC:=\zg_{\tt N}\circ F\circ\tilde{C}$$ is a well-defined functor from $\tt M$ to $\tt Ho(N)\,$. Indeed, let $T_1:=\tilde{C}_1f$ and $T_2:=\tilde{C}_2f$ be two different liftings. The fact that these $\tt M$-morphisms from $A:=\tilde{C}X\in{\tt M}_{\op{c}}$ to $B:=\tilde{C}Y\in{\tt M}_{\op{c}}$ are left homotopic means that $T_1\amalg T_2:A\amalg A\to B$ factors through a cylinder object $\op{Cyl}(A)\,,$ i.e., means that there is a factorization \be\label{WDA}w\circ i_1:=w\circ i\circ \zf_1=\id_A\quad\text{and}\quad w\circ i_2:=w\circ i\circ\zf_2=\id_A\;,\ee where $\zf_1,\zf_2:A\to A\amalg A\,,$ $i:A\amalg A\rightarrowtail\op{Cyl}(A)$ and $w:\op{Cyl}(A)\stackrel{\sim}{\to} A\,,$ as well as a factorization \be\label{WDB}H\circ i_1=T_1\quad\text{and}\quad H\circ i_2=T_2\;,\ee where $H:\op{Cyl}(A)\to B$ (see proof of Theorem \ref{KanWeak}). As the cylinder of a cofibrant object is cofibrant (see proof of Corollary \ref{BasicCor}), we get from $w\circ i_1=\id_A$ that $i_1:A\to\op{Cyl}(A)$ is a weak equivalence between cofibrant objects. If we apply $\zg_{\tt N}\circ F\in\tt Fun(M,Ho(N))$ to \eqref{WDA} and remember that $\zg_{\tt N}(F(i_1))$ is an isomorphism, we see that $\zg_{\tt N}(F(w))$ is the inverse isomorphism and that $\zg_{\tt N}(F(i_1))=\zg_{\tt N}(F(i_2))\,.$ Hence it follows from \eqref{WDB} that $$\zg_{\tt N}(F(\tilde{C}_1f))=\zg_{\tt N}(F(T_1))=\zg_{\tt N}(F(T_2))=\zg_{\tt N}(F(\tilde{C}_2f))\;,$$ so that $\cC f:=\zg_{\tt N}(F(\tilde{C}f))$ is well-defined. Moreover it is now easy to check that $\cC$ respects compositions and identities and is therefore a genuine functor $\cC\in\tt Fun(M,Ho(N))\,$. Finally the diagram \eqref{CTilde} allow us to see that $\tilde{C}f$ is a weak equivalence between cofibrant objects if $f$ is a weak equivalence, so that $\cC$ sends weak equivalences to isomorphisms and factors uniquely through $\tt Ho(M):$ there is a unique functor $\op{Ho}(\cC)$ or \be\label{ExistKDF3}\op{Ho}(F\circ\tilde{C})\in\tt Fun(Ho(M),Ho(N))\;,\ee such that \be\label{ExistKDF4}\op{Ho}(F\circ\tilde{C})\circ \zg_{\tt M}=\zg_{\tt N}\circ F\circ \tilde{C}\;.\ee

From Equations \eqref{ExistKDF1} and \eqref{ExistKDF2} we know that if $Q$ is a {\it cofibrant F-replacement functor}, the functor \be\label{ExistKDF5}\mathbb{L}_Q^{\op{S}}F=\op{Ho}(F\circ Q)\in\tt Fun(Ho(M),Ho(N))\;\ee satisfies the commutation relation \be\label{ExistKDF6}\op{Ho}(F\circ Q)\circ \zg_{\tt M}=\zg_{\tt N}\circ F\circ Q\;.\ee

In the following $\cQ$ denotes both, the replacement $\tilde{C}$ and the replacement $Q\,.$ We now show that $\op{Ho}(F):=\op{Ho}(F\circ\cQ)$ is the right Kan extension $\mathbb{L}^{\op{K}}F$ of $\zg_{\tt N}\circ F$ along $\zg_{\tt M}\,.$\medskip

First we construct a natural transformation $$\ze:\op{Ho}(F)\circ\zg_{\tt M}\Rightarrow\zg_{\tt N}\circ F\;,$$ i.e., a family $\ze_X\,,$ $X\in\tt M\,,$ of $\tt Ho(N)$-maps $$\ze_X:F(\cQ X)\to F(X)$$ that is natural in $X\,.$ Denoting the trivial fibration $$c_X:\tilde{C}X\stackrel{\sim}{\twoheadrightarrow}X\quad\text{or}\quad q_X:QX\stackrel{\sim}{\twoheadrightarrow}X\;$$ by $$\zvf_X:\cQ X\stackrel{\sim}{\twoheadrightarrow}X\;,$$ we get a $\tt Ho(N)$-map $$\zg_{\tt N}(F(\zvf_X)):F(\cQ X)\to F(X)\;$$ and set \be\label{Defze}\ze_X:=\zg_{\tt N}(F(\zvf_X))\;.\ee From the commutation of the lower triangle in \eqref{CTilde} and the naturality of $q$ it follows that \be\label{Natzvf}\zvf_Y\circ\cQ f=f\circ\zvf_X\;,\ee so that the transformation $\ze$ is natural.\medskip

It remains to prove that the pair $(\op{Ho}(F),\ze)$ is universal. Let $H\in\tt Fun(Ho(M),Ho(N))$ and $\zh:H\circ\zg_{\tt M}\Rightarrow\zg_{\tt N}\circ F$ be another such pair. We will show that there is a unique natural transformation $\zk:H\Rightarrow\op{Ho}(F)$ such that $\zh=\ze\circ(\zk\star\zg_{\tt M})\,.$\medskip

In order to define $X$-natural $\tt Ho(N)$-maps $$\zk_{X}:H(X)\to \op{Ho}(F)(X)\,,\quad\text{where}\quad \op{Ho}(F)(X)=F(\cQ X)\;$$ such that $\zh_X=\zg_{\tt N}(F(\zvf_X))\,\circ\,\zk_X\,,$ we consider the $X$-natural $\tt Ho(N)$-maps $$\zh_X:H(X)\to F(X)\;$$ and their naturality square for $\zvf_X:\mathcal{Q}X\to X\,:$

\begin{equation}
\begin{tikzpicture}
 \matrix (m) [matrix of math nodes, row sep=3em, column sep=3em]
   {H(\cQ X) & & F(\cQ X)   \\
    H(X) & & F(X) \\};
 \path[->]
 (m-1-1) edge [->] node[auto] {\small{$\;\;\zh_{\cQ X}$}} (m-1-3)
 (m-1-1) edge [->] node[left] {\small{$H(\zg_{\tt M}(\zvf_X))$}} (m-2-1)
 (m-2-1) edge [->] node[auto] {\small{$\;\;\zh_X$}} (m-2-3)
 (m-1-3) edge [->] node[right] {\small{$\zg_{\tt N}(F(\zvf_X))$}}(m-2-3)
 (m-2-1) edge [->, dashed] node[auto] {\small{$\zk_{X}$}} (m-1-3);
\end{tikzpicture}
\end{equation}
Since $\zvf_X$ is a weak equivalence, the map $H(\zg_{\tt M}(\zvf_X))$ is an isomorphism and we can define $\zk_{X}$ by \be\label{Defzk}\zk_{X}:=\zh_{\cQ X}\circ (H(\zg_{\tt M}(\zvf_X)))^{-1}\;.\ee Of course, since the square commutes, the lower triangle also commutes, which means that $\zh=\ze\circ(\zk\star\zg_{\tt M})$ as already mentioned above.\medskip

In view of Lemma \ref{NatTraMHoM}, the family \eqref{Defzk} of $\tt Ho(N)$-maps defines a natural transformation $\zk:H\Rightarrow \op{Ho}(F)$ if and only if it defines a natural transformation $$\zk:H\circ\zg_{\tt M}\Rightarrow \op{Ho}(F)\circ\zg_{\tt M}\;,\quad\text{where}\quad \op{Ho}(F)\circ \zg_{\tt M}=\zg_{\tt N}\circ F\circ \cQ\;.$$ Let $f:X\to Y$ be an $\tt M$-morphism. If we apply the naturality of the transformation $\zh:H\circ\zg_{\tt M}\Rightarrow\zg_{\tt N}\circ F$ to the morphism $\cQ f:\cQ X\to\cQ Y\,,$ we get $$\op{Ho}(F)(\zg_{\tt M}f)\circ\zk_X=\zg_{\tt N}(F(\cQ f))\circ \zh_{\cQ X}\circ (H(\zg_{\tt M}(\zvf_X)))^{-1}=\zh_{\cQ Y}\circ H(\zg_{\tt M}(\cQ f))\circ (H(\zg_{\tt M}(\zvf_X)))^{-1}\;.$$ Since the commutation relation \eqref{Natzvf} gives $$H(\zg_{\tt M}(\zvf_Y))\circ H(\zg_{\tt M}(\cQ f))=H(\zg_{\tt M}f)\circ H(\zg_{\tt M}(\zvf_X))\;,$$ we finally find that \be\label{Natzk}\op{Ho}(F)(\zg_{\tt M}f)\circ\zk_X=\zk_Y\circ H(\zg_{\tt M}f)\;.\ee

We still have to show that $\zk$ is unique. If $\zk:H\Rightarrow\op{Ho}(F)$ such that $\zh_X=\ze_X\circ\zk_X=\zg_{\tt N}(F(\zvf_X))\circ\zk_X$ exists and if $X\in\tt M_{\op{c}}\,,$ then $\zg_{\tt N}(F(\zvf_X))$ is an isomorphism and $\zk_X$ is necessarily given by \be\label{zkCof}\zk_X=(\zg_{\tt N}(F(\zvf_X)))^{-1}\circ\zh_X=\zh_{\cQ X}\circ (H(\zg_{\tt M}(\zvf_X)))^{-1}\;.\ee If $X\notin{\tt M}_{\op{c}}\,,$ Equation \eqref{Natzk} applied to $$(f:X\to Y)=(\zvf_X:\cQ X\to X)$$ gives $$\zk_X=\zg_{\tt N}(F(\cQ\zvf_X))\circ\zk_{\cQ X}\circ (H(\zg_{\tt M}(\zvf_X)))^{-1}\;.$$ Equation \eqref{zkCof} applied to $\cQ X$ leads to \be\label{Uni1}\zk_X=\zg_{\tt N}(F(\cQ\zvf_X))\circ\zh_{\cQ(\cQ X)}\circ (H(\zg_{\tt M}(\zvf_{\cQ X})))^{-1}\circ (H(\zg_{\tt M}(\zvf_X)))^{-1}\ee and the naturality equation of $\zh$ applied to $$(f:X\to Y)=(\cQ\zvf_X:\cQ(\cQ X)\to \cQ X)$$ shows that \be\label{Uni2}\zg_{\tt N}(F(\cQ\zvf_X))\circ \zh_{\cQ(\cQ X)}=\zh_{\cQ X}\circ H(\zg_{\tt M}(\cQ\zvf_X))\;.\ee From \eqref{Natzvf} applied to $$(f:X\to Y)=(\zvf_X:\cQ X\to X)$$ it follows that $$H(\zg_{\tt M}(\zvf_X))\circ H(\zg_{\tt M}(\cQ\zvf_X))=H(\zg_{\tt M}(\zvf_X))\circ H(\zg_{\tt M}(\zvf_{\cQ X}))\;,$$ so that \be\label{Uni3}H(\zg_{\tt M}(\cQ\zvf_X))=H(\zg_{\tt M}(\zvf_{\cQ X}))\;.\ee If we combine \eqref{Uni1}, \eqref{Uni2} and \eqref{Uni3}, we finally get that $$\zk_X=\zh_{\cQ X}\circ (H(\zg_{\tt M}(\zvf_X)))^{-1}\;,$$ which proves that $\zk$ is unique.
\end{proof}

The previous proof allows us to complete Proposition \ref{ExistKDerFun} as follows:

\begin{theo}\label{Fundamental0}
If $F\in\tt Fun(M,N)$ sends weak equivalences between cofibrant objects to weak equivalences, the Kan extension derived functor $$\mathbb{L}^{\op{K}}F\in\tt Fun(Ho(M),Ho(N))$$ exists and is given by $\op{Ho}(F\circ\tilde{C})$ and by $$\mathbb{L}_Q^{\op{S}}F=\op{Ho}(F\circ Q)\;,$$ where $\tilde{C}$ is a local cofibrant F-replacement and $Q$ is a cofibrant F-replacement functor. This means that \be\label{Fundamental1}\mathbb{L}^{\op{K}}F=\op{Ho}(F\circ\tilde{C})=\mathbb{L}_Q^{\op{S}}F=\op{Ho}(F\circ Q)\;\ee and implies that \be\label{Fundamental2}\mathbb{L}^{\op{K}}F\circ\zg_{\tt M}=\zg_{\tt N}\circ F\circ\tilde{C}=\mathbb{L}_Q^{\op{S}}F\circ\zg_{\tt M}=\zg_{\tt N}\circ F\circ Q\;.\ee
\end{theo}

If we denote \be\label{ompDerFunE2}\mathfrak{i}:\mathbb{L}^{\op{S}}F\stackrel{\cong}{\Rightarrow}\mathbb{L}^{\op{S}}_QF=\op{Ho}(F\circ Q)\ee the natural transformation \eqref{CompDerFunE1} in Proposition \ref{CompDerFunP1}, Equation \eqref{ompDerFunE2} and Equation \eqref{Fundamental2} imply that for every $X\in\tt M\,,$ the value at $\zg_{\tt M}X=X\in\tt Ho(M)$ of the derived functor is $$\mathbb{L}^{\op{S}}F(X)\stackrel{\mathfrak{i}_X}{\cong}\mathbb{L}_Q^{\op{S}}F(X)=\mathbb{L}^{\op{K}}F(X)=\zg_{\tt N}(F(Q
X))=\zg_{\tt N}(F(\tilde{C}X))=F(QX)=F(\tilde{C}X)\;.$$ Further, for every $f\in\op{Hom}_{\tt M}(X,Y)\,,$ the value at $\zg_{\tt M}f\in\op{Hom}_{\tt Ho(M)}(X,Y)$ of the derived functor is
$$\mathbb{L}_Q^{\op{S}}F(\zg_{\tt M}f)=\mathbb{L}^{\op{K}}F(\zg_{\tt M}f)=\zg_{\tt N}(F(Q f))=\zg_{\tt N}(F(\tilde{C}f))\;$$ and $$\mathbb{L}^{\op{S}}F(\zg_{\tt M}f)=\mathfrak{i}^{-1}_Y\circ\mathbb{L}_Q^{\op{S}}F(\zg_{\tt M}f)\circ\mathfrak{i}_X\;.$$

\section{Future directions}\label{FutDir}

Since the discovery of general relativity, the prevailing tendency in mathematics has again been to favor coordinate-independent approaches to problems, as was inevitable in the pre-Descartes era. In particular, the Vinogradov school proposed a coordinate-free cohomological analysis of partial differential equations (PDE-s) \cite{Vino}, an endeavor also promoted in the setting of algebraic geometry by Beilinson and Drinfeld \cite{BD04}. Other authors, e.g., Costello and Gwilliam \cite{CG1}, Schreiber... have investigated a covariant Batalin-Vilkovisky (BV) formalism for gauge field theories. In a series of papers \cite{KTRCR, HAC, PP}, Di Brino and two of the authors of the present article have proposed a generalization to differential operators $\cD$ of homotopical algebraic geometry in the sense of \cite{TV05, TV08} as a suitable framework for the moduli space of solutions of a system of PDE-s modulo symmetries. Indeed, the new geometry in particular provides a convenient method of encoding total derivatives and leads to a covariant description of the classical BV complex which arises as a specific case of general constructions. Further evidence for this standpoint appears in \cite{P11, Paugam1}.\medskip

The mathematically rigorous implementation of the previous ideas requires that the tuple $({\tt DG\cD M, DG\cD M, DG\cD A}, \tau, \mathbf{P})$ be a homotopical algebraic geometric context (HAGC) in the sense of \cite{TV08}. Here $\tt DG\cD M$ is the symmetric monoidal model category of differential graded $\cD$-modules, the subcategory $\tt DG\cD A$ is the model category of differential graded $\cD$-algebras, $\zt$ is a suitable model pre-topology on the opposite category of $\tt DG\cD A$ and $\mathbf{P}$ is a compatible class of morphisms. We expect the proof of the HAGC theorem to be based on a generalization of the concept of homotopy fiber sequence and of Puppe's long exact sequence. In this derived setting, different types of derived functors on model categories are used and need to be compared. We are convinced that the present paper enables us to prove the HAGC theorem and thus to take an important step towards the full implementation of the program described above.

{}
\vfill
{\emph{email:} {\sf alisa.govzmann@uni.lu}; \emph{email:} {\sf damjan.pistalo@uni.lu}; \emph{email:} {\sf norbert.poncin@uni.lu}.}

\end{document}